\documentclass[12pt]{amsart}

\usepackage{latexsym, amssymb, amsmath,ulem,soul,esint}

\usepackage{amsfonts, graphicx}
\usepackage{graphicx,color}
\usepackage{url}
\usepackage{graphicx}
\usepackage{epstopdf}

\newcommand{\be}{\begin{equation}}
\newcommand{\ee}{\end{equation}}
\newcommand{\beq}{\begin{eqnarray}}
\newcommand{\eeq}{\end{eqnarray}}

\newtheorem{thm}{Theorem}[section]

\newtheorem{lma}{Lemma}[section]
\newtheorem{prop}{Proposition}[section]
\newtheorem{cor}{Corollary}[section]

\newtheorem{claim}{Claim}[section]
\theoremstyle {definition}
\newtheorem{defn}{Definition}[section]

\theoremstyle{remark}
\newtheorem{rem}{Remark}[section]

\numberwithin{equation}{section}
\def\tr{\operatorname{tr}}

\def\be{\begin{equation}}
\def\ee{\end{equation}}
\def\bee{\begin{equation*}}
\def\eee{\end{equation*}}

\newcommand{\de}{\partial}

\newcommand{\Ric}{\mathrm{Ric}}

\newcommand{\vp}{\varphi}

\def\supp{\mathrm{supp}}

\def\Ric{\text{\rm Ric}}

\def\tr{\mathrm{tr}}

\def\W{\mathcal{W}}

\def\Div{\mathrm{div}}
\def\ve{\varepsilon}
\def\Hess{\mathrm{Hess}}
\def\id{\mathrm{id}}

\begin{document}

\title[Positive weighted mass theorem]{A non-spin method to the positive weighted mass theorem for weighted manifolds}

\author[J. Chu]{Jianchun Chu$^1$}
\address[Jianchun Chu]{School of Mathematical Sciences, Peking University, Yiheyuan Road 5, Beijing, P.R.China, 100871}
\email{jianchunchu@math.pku.edu.cn}

\author[J. Zhu]{Jintian Zhu$^2$}
\address[Jintian Zhu]{Beijing International Center for Mathematical Research, Peking University, Yiheyuan Road 5, Beijing, P.R.China, 100871}
\email{zhujt@pku.edu.cn, zhujintian@bicmr.pku.edu.cn}

\thanks{$^1$Research partially supported by the Fundamental Research Funds for the Central Universities, Peking University.}

\thanks{$^2$Research partially supported by China postdoctoral science foundation (grant no. BX2021013).}


\date{\today}

\begin{abstract}
In this paper, we investigate the weighted mass for weighted manifolds. By establishing a version of density theorem and generalizing Geroch conjecture in the setting of $P$-scalar curvature, we are able to prove the positive weighted mass theorem for weighted manifolds, which generalizes the result of Baldauf-Ozuch \cite{BaldaufOzuch2022} to non-spin manifolds.
\end{abstract}


\maketitle

\section{Introduction}

The famous positive mass theorem was proved by Schoen-Yau \cite{SchoenYau1979a,SchoenYau1981,SchoenYau2017}, which asserts that the Arnowitt-Deser-Misner (ADM) mass of each end of an asymptotically flat (AF) $n$-manifold with non-negative scalar curvature must be non-negative and the mass vanishes exactly when it is isometric to the standard Euclidean space. Under additional spin assumption, the positive mass theorem was also proved by Witten \cite{Witten1981} using spinor method.

The notion of weighted manifold was first introduced by Lichnerowicz \cite{Lichnerowicz1970,Lichnerowicz197172}. A weighted manifold (also called manifold with density) means a smooth Riemannian manifold $(M,g)$ endowed with a weighted measure $e^{-f}\mathrm{dvol}_{g}$, where $f$ is a smooth function on $M$ and $\mathrm{dvol}_{g}$ is the volume element of $(M,g)$. In \cite{Perelman2002}, Perelman regarded the Ricci flow as a gradient flow of the following functional
\[
\begin{split}
\mathcal{F}(M,g,f)
&= \int_{M}\left(R(g)+|\mathrm{d}f|_{g}^{2}\right)e^{-f}\mathrm{dvol}_{g}\\
&= \int_{M}\left(R(g)+2\Delta_g f-|\mathrm d f|_g^2\right){e^{-f}}\mathrm{dvol}_{g},
\end{split}
\]
where $R(g)$ is denoted to be the scalar curvature of $(M,g)$. This leads to the notion of weighted scalar curvature (or $P$-scalar curvature):
\[
P(g,f) = R(g)+2\Delta_g f-|\mathrm d f|_g^2.
\]
Weighted manifolds with positive or non-negative $P$-scalar curvature have very similar properties compared to Riemannian manifolds with positive or non-negative scalar curvature. By the Gauss-Bonnet theorem and the divergence theorem, every oriented closed surface with positive $P$-scalar curvature must be a topological sphere. For the three-dimensional case, Schoen-Yau \cite{SchoenYau1979b} showed that if $(M^{3},g)$ is oriented and has positive scalar curvature, then it contains no closed, immersed stable minimal surfaces of positive genus. Fan \cite{Fan2008} generalized this result to the weighted case (see \cite{Ho2010} for subsequent work). For further works of the $P$-scalar curvature supporting the above philosophy, we refer the reader to \cite{AbedinCorvino2017,Deng2021,LiMantoulidis2021}.

Recently, there has been interest in establishing the positive weighted mass theorems for weighted manifolds. In \cite{BaldaufOzuch2022}, Baldauf-Ozuch proved the positive weighted mass theorem for spin AF weighted manifolds by generalizing Witten's spinor argument \cite{Witten1981} to the weighted case.
The goal of this paper is to deal with the non-spin case.

To state our main theorem, we have to introduce AF weighted manifolds of $(p,\tau)$-type first. For readers first getting touch with weighted spaces and weighted analysis, we recommend them to consult Appendix \ref{Sec: Appendix A} before going further.
Let
$$
p>n\mbox{ and }\tau\in\left(\frac{n-2}{2},n-2\right).
$$

\begin{defn}\label{Defn: AF}
A weighted manifold $(M^n,g,f)$ is said to be AF of $(p,\tau)$-type if $n\geq 3$ and there is a compact subset $K\subset M$ such that
\begin{itemize}\setlength{\itemsep}{1mm}
\item[(a)] the complement $M-K$ consists of finitely many ends $E_1, E_2,\ldots, E_N$ and each $E_k$ admits a diffeomorohism $\Phi_k:E_k\to \mathbb{R}^{n}\setminus B_{1}$;
\item[(b)] in coordinate chart $(E_k,\Phi_k)$, the components $g_{ij}-\delta_{ij}$ and the function $f$ belong to $W^{2,p}_{-\tau}(E_k)$;
\item[(c)] the scalar curvature $R(g)$ and the Laplaican $\Delta_g f$ belong to $L_{-2\tau-2}^{\infty}$;
\item[(d)] the $P$-scalar curvature $P(g,f)$ belongs to $L^1$;
\item[($\ast$)] when $n=3$, we further assume the technical condition that $g_{ij}-\delta_{ij}$ belongs to $W^{1,\infty}_{-1}(E_k)$ on each end $E_k$.
\end{itemize}
\end{defn}

Recall the following weighted mass for AF weighted manifolds from \cite{BaldaufOzuch2022}.

\begin{defn}
Given an AF weighted manifold $(M^n,g,f)$ of $(p,\tau)$-type, the weighted mass of an end $E$ is defined to be
\[
\begin{split}
m(M,g,f,E) = {} & 2(n-1)\omega_{n-1}\cdot m_{\mathrm{ADM}}(M,g,E)+2\lim_{\rho\to \infty} \int_{S_\rho} \partial_i f \cdot \nu^i \cdot e^{-f}\, \mathrm{d}\sigma,
\end{split}
\]
where $\omega_{n-1}$ is the volume of the unit sphere in $\mathbb{R}^{n}$, $S_\rho$ is the coordinate $\rho$-sphere $\{x\in \mathbb{R}^n: |x|=\rho\}$, $\nu$ and $\mathrm{d}\sigma$ are the outward unit normal and volume form of $S_\rho$ in $\mathbb R^n$.
\end{defn}

\begin{rem}
For later use, we point out an equivalent formula of the weighted mass:
\begin{equation}\label{equivalent formula weighted mass}
m(M,g,f,E) = \lim_{\rho\to \infty} \int_{S_\rho} (\partial_jg_{ij}-\partial_ig_{jj}+2\partial_if)\, \nu^i \, \mathrm{d}\sigma.
\end{equation}
Indeed, by $f\in W_{-\tau}^{2,p}$ and Theorem \ref{weighted inequalities} (iii), we have $f\in W_{-\tau}^{1,\infty}$ and so
\[
\int_{S_\rho}\partial_i f \cdot \nu^i \cdot (e^{-f}-1)\,\mathrm{d}\sigma = O(\rho^{n-1}\cdot \rho^{-\tau-1}\cdot \rho^{-\tau}) = O(\rho^{n-2-2\tau}).
\]
Since we have $\tau>\frac{n-2}{2}$, then
\[
\lim_{\rho\to \infty}\int_{S_\rho}\partial_i f \cdot \nu^i \cdot (e^{-f}-1)\,\mathrm{d}\sigma = 0,
\]
which implies \eqref{equivalent formula weighted mass}.
\end{rem}

Our main theorem is as follows.
\begin{thm}\label{Thm: main}
Let $3\leq n\leq 7$. If $(M^n,g,f)$ is an AF weighted manifold of $(p,\tau)$-type with non-negative $P$-scalar curvature, then for each end $E$ we have
\begin{equation}\label{Eq: mass inequality}
m(M,g,f,E)\geq 0.
\end{equation}
If the equality holds for some end $E$, then $(M,g)$ is isometric to the Euclidean $n$-space $(\mathbb R^n,g_{\mathbb{E}})$ and $f$ is the zero function.
\end{thm}

\begin{rem}
When $(M^{n},g)$ is spin and the weighted mass equals to zero, Baldauf-Ozuch \cite{BaldaufOzuch2022} showed that $(M^{n},g)$ is isometric to $(\mathbb R^n,g_{\mathbb{E}})$ and
\[
\int_{\mathbb{R}^{n}}(\Delta_{\mathbb{E}}f-|\mathrm{d}f|_{g_{\mathbb{E}}}^{2})\,e^{-f}\mathrm{d}x = 0.
\]
Compared to this result, our rigidity for $f$ is slightly stronger.
\end{rem}

Now let us discuss the proof of Theorem \ref{Thm: main}. Basically, we follow the standard arguments from the proof of Riemannian positive mass theorem for (unweighted) AF manifolds. In detail, we first establish a density theorem that any AF weighted manifold can be modified to have harmonically flat ends without changing weighted mass dramatically. Second we apply the compactification trick of Lohkamp \cite{Lohkamp1999} to reduce the desired mass inequality to the existence of topological obstruction for positive $P$-scalar curvature on Schoen-Yau-Schick manifolds, which can be then verified by using weighted stable slicing technique in \cite{BrendleHirschJohne2022} originating from \cite{SchoenYau2017}. In our way to show the density theorem we have to deal with the conformal operator
$$T_{g,f}(u)=P(u^{\frac{4}{n-2}}g,f)=u^{-\frac{n+2}{n-2}}\left(-\frac{4(n-1)}{n-2}\Delta_g u +4\langle\mathrm{d}u,\mathrm{d}f\rangle_g+P(g,f)u\right),$$
where the extra gradient term makes it impossible to carry on the original arguments from \cite{SchoenYau1981} to find a solution of $P(u^{\frac{4}{n-2}}g,f)=0$ based on Fredholm alternative. Our idea to get rid of this difficulty is inspired by \cite{EHLS2016}. Namely we investigate the conformal operator coupled with the full deformation operator $S_{g,f}(h)=P(g+h,f)$. As a result, possibly non-trivial kernel of $DT_{g,f}$ causing the failure of Fredholm alternative is now compensated by the surjectivity of $DS_{g,f}$.

The organization of this paper is as follows. In Section \ref{Sec: Density} and \ref{Sec: Geroch}, we will establish the density theorem and prove the generalized Geroch conjecture. In Section \ref{Sec: Proof}, we will give the proof of Theorem \ref{Thm: main}. In Appendix \ref{Sec: Appendix A}, we will collect some basic notions and results of weighted analysis used in this paper.

\section{Density theorem}\label{Sec: Density}
For convenience, throughout this paper we denote
$$
s=\frac{4}{n-2}
$$
and call $(g,f)$ a {\it smooth pair} if $g$ and $f$ are a smooth metric and a smooth function on $M$ respectively.

The purpose of this section is to prove the following density theorem.
\begin{prop}\label{Density}
Let $(M,g,f)$ be an AF weighted manifold of $(p,\tau)$-type with non-negative $P$-scalar curvature. For any constant $\ve>0$, we can construct a new smooth pair $(g_\ve,f_\ve)$ such that
\begin{itemize}\setlength{\itemsep}{1mm}
\item $(M,g_\ve, f_\ve)$ is still an AF weighted manifold of $(p,\tau)$-type with non-negative $P$-scalar curvature;
\item for each end $E_k$, there is a constant $r_k$ such that
$$
g_\ve=u_{\ve,k}^s g_{\mathbb{E}}\mbox{ and }f_\ve\equiv 0\mbox{ when }|x|\geq r_k,
$$
where $u_{\ve,k}$ is a harmonic function with the expression
$$
u_{\ve,k}=1+\frac{A_{\ve,k}}{r^{n-2}}+O(r^{1-n});
$$
\item we have
$$
|m(M,g,f,E_k)-m(M,g_\ve,f_\ve,E_k)|<\ve.
$$
\end{itemize}
\end{prop}

\subsection{Operators and their linearizations}

In order to distinguish weighted spaces of tensors from that of functions, we write
\[
\W_{-\tau}^{2,p} = \{\text{$h$ is a symmetric $(0,2)$-tensor on $M$ with $W_{-\tau}^{2,p}$-components}\}.
\]
For later use, we define
\begin{equation}\label{Eq: T}
T_{g,f}:\left (1+W_{-\tau}^{2,p}\right) \rightarrow L_{-\tau-2}^{p}, \ \ T_{g,f}(u) = P(u^{s}g,f),
\end{equation}
\begin{equation}\label{Eq: S}
S_{g,f}: \W_{-\tau}^{2,p}\rightarrow L_{-\tau-2}^{p}, \ \ S_{g,f}(h) = P(g+h,f),
\end{equation}
\begin{equation}\label{Eq: Phi}
\Phi_{g,f}: (1+W_{-\tau}^{2,p})\times\W_{-\tau}^{2,p}  \rightarrow L_{-\tau-2}^{p}, \ \ \Phi_{g,f}(u,h)= P(u^{s}g+h,f).
\end{equation}
Notice that all these operators are defined by changing input tensor of $P$-scalar curvature. Since the input tensor of $P$-scalar curvature is always required to be positive-definite, these operators are actually well-defined only in small $W^{2,p}_{-\tau}$- or $\mathcal W^{2,p}_{-\tau}$-neighborhoods. Since no confusion seems to be caused on this point, we just omit the accurate description of these neighborhoods for convenience.

As preparation it is straightforward to show
\begin{lma}\label{Lem: linearization}
We have
\begin{equation}\label{DT}
DT_{g,f}\big|_{u=1}(v) = -(n-1)s\Delta_g v +4\langle\mathrm{d}v,\mathrm{d}f\rangle_g-sP(g,f)v, \\[1mm]
\end{equation}
\begin{equation}\label{DS}
\begin{split}
DS_{g,f}&\big|_{h=0}(\xi) = {}  -\Delta_g(\tr_g \xi)+\Div_g\Div_g \xi-\langle \xi,\Ric(g)\rangle_g\\[1mm]
&-2\langle \xi,\Hess_g f\rangle_g  -2\langle\Div_g \xi,\mathrm df\rangle_g+\langle \mathrm d\tr_g \xi,\mathrm df\rangle_g +\langle \xi,\mathrm df\otimes \mathrm df\rangle_g,\\[1mm]
\end{split}
\end{equation}
\begin{equation}\label{DPhi}
D\Phi_{g,f}\big|_{(u,h)=(1,0)}(v,\xi) = DT_{g,f}\big|_{u=1}(v)+DS_{f}\big|_{h=g}(\xi).
\end{equation}
\end{lma}

\begin{proof}
Since \eqref{DPhi} is trivial, it suffices to prove \eqref{DT} and \eqref{DS}. For \eqref{DT}, it is well-known that
\[
R(u^sg) = u^{-(1+s)}\big(R(g)u-(n-1)s\Delta_{g}u\big)
\]
and
\[
\Delta_{u^{s}g}f = u^{-s}\Delta_{g}f+2s^{-1}u^{-2s}\langle\mathrm{d}u^{s},\mathrm{d}f\rangle_{g}.
\]
Then for any $v\in W_{-\tau}^{2,p}$,
\[
\begin{split}
& T_{g,f}(1+tv)\\[1.5mm]
= {} & (1+tv)^{-(1+s)}\big(R(g)(1+tv)-(n-1)s\Delta_{g}(1+tv)\big) \\[1.5mm]
+ \  & 2(1+tv)^{-s}\Delta_{g}f+4s^{-1}(1+tv)^{-2s}\langle\mathrm{d}(1+tv)^{s},\mathrm{d}f\rangle_{g}-(1+tv)^{-s}|\mathrm df|_{g}^{2},
\end{split}
\]
which implies
\[
\left.DT_{g,f}\right|_{u=1}(v) = \frac{d}{dt}\bigg|_{t=0}T_{g,f}\big(1+tv\big)
= -(n-1)s\Delta_g v +4\langle\mathrm{d}v,\mathrm{d}f\rangle_g-sP(g,f)v,
\]
as required.

\smallskip

For \eqref{DS}, recall from \cite[Remark 3.8 and Lemma 3.2]{ChowKnopf2004} that
\begin{equation}\label{DS eqn 1}
\left.\frac{\mathrm d}{\mathrm d t}\right|_{t=0}R(g+t\xi) = -\Delta_g(\tr_g \xi)+\Div_g\Div_g \xi-\langle \xi,\Ric(g)\rangle_g
\end{equation}
and
$$
\left.\frac{\mathrm d}{\mathrm d t}\right|_{t=0}\Gamma_{ij}^k(g+t\xi)=\frac{1}{2}g^{kl}(\nabla_i \xi_{jl}+\nabla_j \xi_{il}-\nabla_l \xi_{ij}).
$$
Then we compute
\begin{equation}\label{DS eqn 2}
\begin{split}
\left.\frac{\mathrm d}{\mathrm d t}\right|_{t=0}\Delta_{g+t\xi}f
= {} & \left.\frac{\mathrm d}{\mathrm d t}\right|_{t=0}(g+t\xi)^{ij}\left(f_{ij}-\Gamma_{ij}^{k}(g+t\xi)\cdot f_{k}\right)\\
= {} &-\langle \xi,\Hess_g f\rangle_g-g^{ij}\left.\frac{\mathrm d}{\mathrm d t}\right|_{t=0}\Gamma_{ij}^k(g+t\xi)\cdot f_k\\
= {} & -\langle \xi,\Hess_g f\rangle_g-\langle\Div_g \xi,\mathrm df\rangle_g+\frac{1}{2}\langle \mathrm d\tr_g \xi,\mathrm df\rangle_g.
\end{split}
\end{equation}
It is also easy to see
\begin{equation}\label{DS eqn 3}
\left.\frac{\mathrm d}{\mathrm d t}\right|_{t=0}|\mathrm df|^2_{g+t\xi}=-\langle \xi,\mathrm df\otimes \mathrm df\rangle_g.
\end{equation}
Then \eqref{DS} follows from \eqref{DS eqn 1}, \eqref{DS eqn 2} and \eqref{DS eqn 3}.
\end{proof}

\subsection{Surjectivity of the operator $DS_{g,f}\big|_{h=0}$}
In this subsection, we are going to prove

\begin{lma}\label{Surjectivity}
The operator $DS_{g,f}\big|_{h=0}:\W_{-\tau}^{2,p} \rightarrow L_{-\tau-2}^{p}$ is surjective.
\end{lma}

\begin{proof}
Let us write $A=DS_{g,f}\big|_{h=0}$ for short and split the argument into the following three steps.

\bigskip
\noindent
{\bf Step 1.} The range of $A$ is closed.
\bigskip

Define an operator $B:W^{2,p}_{-\tau}\rightarrow L_{-\tau-2}^{p}$ by $B(\psi)=A(\psi g)$. Using \eqref{DS}, we see that
\[
B(\psi) = -(n-1)\Delta_{g}\psi+(n-2)\langle \mathrm d\psi,\mathrm df\rangle_g-\left(R(g)+2\Delta_{g}f-|\mathrm{d}f|_{g}^{2}\right)\psi.
\]
Then its $L^2$-adjoint operator $B^{*}:L_{\tau+2-n}^{p}\rightarrow(W_{-\tau}^{2,p})^{*}$ is given by
\begin{equation}\label{Eq: B*}
B^{*}(\phi) = -(n-1)\Delta_{g}\phi-(n-2)\langle \mathrm d\phi,\mathrm df\rangle_g-\left(R(g)+n\Delta_{g}f-|\mathrm{d}f|_{g}^{2}\right)\phi.
\end{equation}
By Theorem \ref{Fredholm}, the operator $B$ is Fredholm and so the range of $B$ has finite codimension in $L_{-\tau-2}^{p}$. Since the range of $A$ contains the range of $B$, it is of finite codimension in $L_{-\tau-2}^{p}$ as well. It follows that the range of $A$ is closed.

\bigskip
\noindent
{\bf Step 2.} Denote the $L^2$-adjoint operator of $A$ by $A^{*}$. For any $\phi\in\mathrm{Ker}A^{*}$, $\phi$ and $\de\phi$ vanish to infinite order at infinity, i.e. $\phi\in W^{1,\infty}_{-N}$ for any positive integer $N$.\bigskip

A direct calculation shows that the $L^2$-adjoint operator $A^{*}:L_{\tau+2-n}^{p}\rightarrow(\W_{-\tau}^{2,p})^{*}$ is given by
\begin{equation}\label{adjoint A}
\begin{split}
A^{*}(\phi) = {} & -(\Delta_g \phi)g+\Hess_g \phi+2\mathrm d\phi\otimes \mathrm df-\langle \mathrm d \phi,\mathrm df\rangle_g g \\[1mm]
& -\phi(\Delta_g f) g+\phi \mathrm df\otimes \mathrm df-\phi\Ric(g).
\end{split}
\end{equation}
For any $\phi\in\mathrm{Ker}A^{*}$, we have $A^{*}(\phi)=0$. Taking trace of both sides of $A^{*}(\phi)=0$, it then follows that
\begin{equation}\label{Laplacian phi}
\begin{split}
\Delta_g \phi = {} & -\frac{1}{n-1}\Big((n-2)\langle\mathrm d\phi,\mathrm df\rangle_{g}+\left(R(g)+n\Delta_{g}f-|\mathrm{d}f|_{g}^{2}\right)\phi\Big).
\end{split}
\end{equation}
Substituting this into \eqref{adjoint A} and using $A^{*}(\phi)=0$, we see that
\begin{equation}\label{Hess phi}
\begin{split}
\Hess_g \phi  &= {}  -\frac{1}{n-1}\Big((n-2)\langle\mathrm d\phi,\mathrm df\rangle_{g}+\left(R(g)+n\Delta_{g}f-|\mathrm{d}f|_{g}^{2}\right)\phi\Big)g \\[2mm]
& -2\mathrm d\phi\otimes \mathrm df+\langle \mathrm d \phi,\mathrm df\rangle_g g+\phi(\Delta_g f) g-\phi \mathrm df\otimes \mathrm df+\phi\Ric(g).
\end{split}
\end{equation}

\begin{claim}\label{vanishing order claim 1}
If $\phi\in L^{2}_{-\mu}$ for some $\mu$, then $\phi\in W^{2,q}_{-\nu}$ for any $q$ and $\nu<\mu$.
\end{claim}

Applying Theorem \ref{weighted inequalities} (iii) and Theorem \ref{elliptic theory} repeatedly to \eqref{Laplacian phi}, we obtain $\phi\in L^{\infty}_{-\mu}$. By Theorem \ref{weighted inequalities} (i), $\phi\in L^{q}_{-\nu}$ for any $q$ and $\nu<\mu$. Using Theorem \ref{elliptic theory} again, we have $\phi\in W^{2,q}_{-\nu}$.

\begin{claim}\label{vanishing order claim 2}
If $\phi\in W^{2,q}_{-\nu}$ for any $q$ and $\nu<\mu$, then $\phi\in L^{2}_{-\nu-\tau}$.
\end{claim}

By \eqref{Hess phi}, we have
\begin{equation}\label{Hess phi 1}
|\Hess_g \phi|_{g} \leq C|\mathrm df|_{g}|\mathrm d\phi|_{g}+C\Big( |R(g)|+|\Ric(g)|+|\Delta_{g}f|+|\mathrm df|_{g}^{2} \Big)|\phi|.
\end{equation}
The assumption $\phi\in W^{2,q}_{-\nu}$ shows
\begin{equation}\label{vanishing order claim 2 eqn 1}
\text{$|\phi| \in L^{q}_{-\nu}$ and $|\mathrm d\phi|_{g} \in L^{q}_{-\nu-1}$ for any $q$}.
\end{equation}
Using the assumption (b) in Definition \ref{Defn: AF} and Theorem \ref{weighted inequalities} (iii), we see that
\begin{equation}\label{vanishing order claim 2 eqn 2}
|R(g)|, \ |\Ric(g)|, \ |\Delta_{g}f| \in L^{p}_{-\tau-2}\mbox{ and }
|\mathrm df|_{g} \in L^{\infty}_{-\tau-1}.
\end{equation}
Substituting \eqref{vanishing order claim 2 eqn 1} and \eqref{vanishing order claim 2 eqn 2} into \eqref{Hess phi 1} and using Theorem \ref{weighted inequalities} (ii), we obtain $|\Hess_g \phi|_{g}\in L^{2}_{-\nu-\tau-2}$. Using the following weighted Poincar\'e inequality on each end $E_k$ (see \cite[(10)]{CorvinoSchoen2006} for $n=3$ and \cite[p.111]{EHLS2016} for general $n$)
\[
\int_{M\setminus B_{R}}\left(\phi\cdot|x|^{\tau+\nu}\right)^{2}|x|^{-n}\mathrm{d}x
\leq C\int_{M\setminus B_{R}}\left(|\Hess_g \phi|_{g}\cdot|x|^{\tau+\nu+2}\right)^{2}|x|^{-n}\mathrm{d}x,
\]
we obtain $\phi\in L^{2}_{-\nu-\tau}$.

\smallskip

Now let us prove Step 2. Recall that $\phi\in L_{\tau+2-n}^{p}$. From Theorem \ref{weighted inequalities} (i), we see that $\phi\in L_{\tau'+2-n}^{2}$ for any $\tau'>\tau$. Using Claim \ref{vanishing order claim 1} and \ref{vanishing order claim 2} repeatedly, we can obtain $\phi\in W^{2,q}_{-N}$ for any $q$ and $N$. By Theorem \ref{weighted inequalities} (iii), we obtain $\phi\in W^{1,\infty}_{-N}$.

\bigskip
\noindent
{\bf Step 3.} $\mathrm{Ker}A^{*}$ is trivial.
\bigskip

Using assumptions (c), $f\in W_{-\tau}^{2,p}$ in Definition \ref{Defn: AF} and Theorem \ref{weighted inequalities} (iii), we have
\[
R(g)\in L_{-2\tau-2}^{\infty},\,\Delta_{g}f\in L_{-2\tau-2}^{\infty}\mbox{ and }|\mathrm df|_{g} \in L^{\infty}_{-\tau-1}.
\]
Fixing an end $E_k$ and using \eqref{Laplacian phi}, we obtain
\begin{equation}\label{Delta g phi}
|\Delta_g \phi| \leq C|x|^{-\tau-1}|\mathrm d\phi|_{g}+C|x|^{-2\tau-2}|\phi|.
\end{equation}
Define a new metric $\bar{g}=|x|^{-4}g$ on $E_k$. In particular, we have
\begin{equation}\label{Eq: Delta bar g}
\Delta_{\bar g}\phi = |x|^{4}\Delta_{g}\phi+2s^{-1}|x|^{8}\langle \mathrm{d}|x|^{-4},\mathrm{d}\phi\rangle_{g}.
\end{equation}
In the following, we consider the function $\bar\phi:=|x|^{n-2}\phi$. Using \eqref{Eq: Delta bar g}, it is easy to compute
\begin{equation}\label{Delta bar g bar phi 1}
\begin{split}
\Delta_{\bar g}\bar\phi&=\Delta_{\bar g} |x|^{n-2}\cdot \phi+2\langle\mathrm d|x|^{n-2},\mathrm d\phi\rangle_{\bar g}+|x|^{n-2}\cdot
\Delta_{\bar g}\phi\\
&=\Delta_{\bar g} |x|^{n-2}\cdot \phi+|x|^{n+2}\Delta_{g}\phi.
\end{split}
\end{equation}
From \eqref{Eq: Delta bar g} and the asymptotical flatness of $g$, we can check $\Delta_{\bar g}|x|^{n-2}=O(|x|^{n-\tau})$. Combining this with \eqref{Delta g phi} and \eqref{Delta bar g bar phi 1}, we conclude
\begin{equation}\label{Delta bar g bar phi 2}
\begin{split}
\left|\Delta_{\bar g}\bar\phi\right|
\leq {} & C|x|^{n+1-\tau}|\mathrm d\phi|_g+C\left(|x|^{n-\tau}+|x|^{n-2\tau}\right)|\phi|\\[0.5mm]
\leq {} & C|x|^{1-\tau}\left|\mathrm d\bar\phi\right|_{\bar g}+C\left(|x|^{2-\tau}+|x|^{2-2\tau}\right)|\bar\phi| \\
\leq {} & C|x|^{1-\tau}\left|\mathrm d\bar\phi\right|_{\bar g}+C|x|^{2-\tau}|\bar\phi|.
\end{split}
\end{equation}
Let us define
$$y=\frac{x}{|x|^{2}}$$ according to the Kelvin transformation and then we work in the new coordinate system $(B_{R^{-1}}\setminus\{0\};\{y_{i}\}_{i=1}^{n})$.
Through a direct computation, we have
\[
\begin{split}
&\bar g\left(\frac{\partial}{\partial y_i},\frac{\partial}{\partial y_j}\right)=g_{ij}\left(\frac{y}{|y|^2}\right)-2|y|^{-2}\left[g_{il}\left(\frac{y}{|y|^2}\right)-\delta_{il}\right]y_jy_l\\
&\qquad\quad-2|y|^{-2}\left[g_{kj}\left(\frac{y}{|y|^2}\right)-\delta_{kj}\right]y_iy_k+4|y|^{-4}\left[g_{kl}\left(\frac{y}{|y|^2}\right)-\delta_{kl}\right]y_ky_ly_iy_j.
\end{split}
\]
It follows from
\begin{equation}\label{Eq: -1 decay}
g_{ij}-\delta_{ij}\in W^{1,\infty}_{-1}(E_k)
\end{equation}
that the metric $\bar g$ can be extended to a Lipschitz metric in $B_{1/R}$,
where we use the fact $\tau>1$ when $n\geq 4$ and the technical assumption $(\ast)$ in Definition \ref{Defn: AF} when $n=3$ to obtain \eqref{Eq: -1 decay} respectively.
Recall from Step 2 that $\phi$ and $\de\phi$ have infinite vanishing order at infinity. As a result, we have
\begin{equation}\label{Eq: infinite order}
\left|\bar \phi(y)\right|+\left|\left(\mathrm d\bar\phi\right)(y)\right|_{\bar g}=O\big(|y|^N\big)\mbox{ for all } N>0.
\end{equation}
In particular, the function $\bar\phi$ belongs to $W^{1,\infty}(B_{1/R})$. Rewriting \eqref{Delta bar g bar phi 2} as
\begin{equation*}
|(\Delta_{\bar{g}}\bar\phi)(y)| \leq C|y|^{\tau-1}|(\mathrm d\bar\phi)(y)|_{\bar{g}}+C|y|^{\tau-2}|\bar\phi(y)|
\end{equation*}
and using \eqref{Eq: infinite order}, we conclude $\Delta_{\bar g}\bar \phi \in L^\infty(B_{1/R})$. Through approximation, one can verify
$$
\int_{B_{1/R}}\Delta_{\bar g}\bar \phi \cdot \eta\,\mathrm{dvol}_{\bar g}=-\int_{B_{1/R}}\langle\mathrm d\bar\phi,\mathrm d\eta\rangle_{\bar g}\,\mathrm{dvol}_{\bar g}\mbox{ for any }\eta \in C^\infty_0(B_{1/R}).
$$
Now we can apply the strong continuation property from \cite[Theorem 1.8 or Lemma 2.4]{Kazdan1988} to deduce $\bar \phi\equiv 0$. As a result, we obtain $\phi\equiv 0$ and so $\mathrm{Ker}A^{*}$ is trivial.

\bigskip

Now it is easy to prove Lemma \ref{Surjectivity}. From Step 1 and 3, we know that $A$ has closed range and that $\mathrm{Ker}A^{*}$ is trivial, which yields the surjectivity of $A$.
\end{proof}

\subsection{Proof of the density theorem}
To prove the density theorem, we need the following lemma on continuity of mass.
\begin{lma}\label{Lem: continuity mass}
Let $(M,g,f)$ and $(M,\hat g,\hat f)$ be two AF weighted manifolds of $(p,\tau)$-type with the same underlying space. Given any $\ve>0$, there exists a constant $\delta>0$ depending only on $\ve$, $n$, $p$, $\tau$,
$$
\|g-g_{\mathbb E}\|_{W^{2,p}_{-\tau}}, \, \|f\|_{W_{-\tau}^{2,p}} \mbox{ and } \|P(g,f)-P(\hat g,\hat f)\|_{L^1_{-2\tau-2}}
$$ such that if
$$
\|\hat g-g\|_{W^{2,p}_{-\tau}}< \delta \mbox{ and }\|\hat f-f\|_{W^{2,p}_{-\tau}} < \delta,
$$
then we have
$$
|m(M,g,f,E_k)-m(M,\hat g,\hat f,E_k)|<\ve\mbox{ for each end }E_k.
$$
\end{lma}
\begin{proof}
Fix one end $E_k$. By \eqref{equivalent formula weighted mass}, we have
\[
\begin{split}
& m(M,g,f,E_k)\\
& = \lim_{\rho\to \infty} \int_{S_\rho} (\partial_jg_{ij}-\partial_ig_{jj}+2\partial_if) \,\nu^i \, \mathrm{d}\sigma \\
& = \int_{S_\rho} (\partial_jg_{ij}-\partial_ig_{jj}+2\partial_if) \,\nu^i \, \mathrm{d}\sigma
+\int_{\{|x|\geq\rho\}} (\de_{i}\partial_jg_{ij}-\partial_{j}\de_{j}g_{ii}+2\Delta_{g}f)\, dx \\[1mm]
& =:  A(g,f,\rho)+B(g,f,\rho),
\end{split}
\]
and the same formula holds for $m(M,\hat g,\hat f,E_k)$. In the rest of the proof, we always denote $C$ to be a constant depending only on $\ve$, $n$, $p$, $\tau$,
$$
\|g-g_{\mathbb E}\|_{W^{2,p}_{-\tau}}, \, \|f\|_{W_{-\tau}^{2,p}}  \mbox{ and } |P(g,f)-P(\hat g,\hat f)|_{L^1_{-2\tau-2}}.
$$
From \cite[(4.2) in page 680]{Bartnik1986}, $f\in W_{-\tau}^{2,p}$ and Theorem \ref{weighted inequalities} (iii), we know
$$
|R(g)-(\de_{i}\partial_jg_{ij}-\partial_{j}\de_{j}g_{ii})|+|df|_{g}^{2} \leq C|x|^{-2\tau-2}.
$$
Since we want to take $\delta$ to be small, we can assume $\delta\leq 1$ a priori. Then we also have
$$
|R(\hat g)-(\de_{i}\partial_j\hat g_{ij}-\partial_{j}\de_{j}\hat g_{ii})|+|d\hat{f}|_{\hat{g}}^{2} \leq C|x|^{-2\tau-2}.
$$
This yields
\[
\begin{split}
&|B(g,f,\rho)-B(\hat g,\hat f,\rho)| \\
\leq {} & C\int_{\{|x|\geq\rho\}}|x|^{-2-2\tau}\,\mathrm dx+\int_{\{|x|\geq\rho\}}|P(g,f)-P(\hat g,\hat f)|\,\mathrm{d}x \\
\leq {} & \left(C+\|P(g,f)-P(\hat g,\hat f)\|_{L^1_{-2\tau-2}}\right)\rho^{n-2-2\tau}.
\end{split}
\]
On the other hand, we have
$$
|A(g,f,\rho)-A(\hat g,\hat{f},\rho)|\leq C\rho^{n-2-\tau}\left(\|\hat g-g\|_{W^{2,p}_{-\tau}}+\|\hat f-f\|_{W^{2,p}_{-\tau}}\right)\leq C\rho^{n-2-\tau}\delta.
$$
The proof is now completed by fixing $\rho$ such that $|B(g,f,\rho)-B(\hat g,\hat f,\rho)|<\ve/2$ and taking $\delta$ small enough such that $|A(g,f,\rho)-A(\hat g,\hat{f},\rho)|<\ve/2$.
\end{proof}

We now in a position to prove the density theorem.

\begin{proof}[Proof of Proposition \ref{Density}]
Denote
$$S_{g,f}: \mathcal W^{2,p}_{-\tau}\to L^p_{-\tau-2} \mbox{ and } T_{g,f}:(1+W^{2,p}_{-\tau})\to L^p_{-\tau-2}$$ to be the operators from \eqref{Eq: S} and \eqref{Eq: T}. Recall from Lemma \ref{Surjectivity} that the operator $DS_{g,f}\big|_{h=0}$ is surjective. From Theorem \ref{Fredholm}, we see that the operator $DT_{g,f}\big|_{u=1}$ is Fredholm and so there exist finitely many compactly supported smooth symmetric $(0,2)$-tensors $\{\chi_{k}\}_{k=1}^{N}\subset \W^{2,p}_{-\tau}$ such that their corresponding images $$\left\{DS_{g,f}\big|_{h=0}(\chi_{k})\right\}_{k=1}^{N}$$ form a basis for a complementing subspace of the range of $DT_{g,f}\big|_{u=1}$ in $L^{p}_{-\tau-2}$.

Let us take $$\mathcal K_1=\mbox{a complementing subspace of }\mathrm{Ker}\left(DT_{g,f}\big|_{u=1}\right)\mbox{ in }W_{-\tau}^{2,p}$$ and $$\mathcal K_{2}=\mathrm{span}\{\chi_{k}\}_{k=1}^{N}\subset \W_{-\tau}^{2,p}.$$
Consider the map $$\Psi_{g,f}=\Phi_{g,f}|_{(1+\mathcal K_1)\times \mathcal K_2},$$
where $\Phi_{g,f}:(1+W_{-\tau}^{2,p})\times\W_{-\tau}^{2,p}\to L^p_{-\tau-2}$ is the operator defined in \eqref{Eq: Phi}.
It follows from Lemma \ref{Lem: linearization} that the corresponding linearized map $$D\Psi_{g,f}\big|_{(u,h)=(1,0)}:\mathcal K_{1}\times \mathcal K_{2}\rightarrow L^{p}_{-\tau-2}$$ is an isomorphism.

Let $\eta:\mathbb R\to [0,1]$ be a smooth cut-off function such that
\[
\text{$\eta\equiv1$ in $(-\infty,1]$ and $\eta\equiv0$ in $[2,\infty)$}.
\]
For any $\lambda\in(-1,1)$, we define $\eta_{\lambda}(x)=\eta(|\lambda||x|)$ and take
\[
g_{\lambda} = \eta_{\lambda}g+(1-\eta_{\lambda})g_{\mathbb E}\mbox{ and }
f_{\lambda} = \eta_{\lambda}f.
\]
It is clear that $(g_{\lambda},f_{\lambda})\rightarrow(g,f)$ in $\W_{-\tau}^{2,p}\times W_{-\tau}^{2,p}$ as $\lambda\rightarrow 0$. Consider the map
$$
G:(-1,1)\times (1+\mathcal K_1)\times \mathcal K_2\to L^p_{-\tau-2}
$$
given by
$$
G(\lambda,u,h)=\Psi_{g_\lambda,f_\lambda}(u,h)-\eta_\lambda P(g,f).
$$
It is not difficult to verify the following properties
\begin{itemize}\setlength{\itemsep}{1mm}
\item $G$ is continuous;
\item $G$ has $F$-derivative $G'$ with respect to $(u,h)$, which is also continuous;
\item we have $G(0,1,0)=0$ and that the derivative
$$G'(0,1,0)=D\Psi_{g,f}\big|_{(u,h)=(1,0)}$$ is invertible.
\end{itemize}
From the implicit function theorem (see \cite[Theorem 1.2.1]{Chang2005} for instance), we conclude that for $\lambda$ small enough, there exists $(u_\lambda,h_\lambda)$ in $(1+\mathcal K_1)\times \mathcal K_2$ such that
\begin{equation}\label{Eq: g lambda}
\Psi_{g_\lambda,f_\lambda}(u_\lambda,h_\lambda)=\eta_\lambda P(g,f).
\end{equation}
Moreover, we have $(u_0,h_0)=(1,0)$ and that the map $\lambda\mapsto (u_\lambda,h_\lambda)$ is $C^1$.

Denote $\hat g_\lambda=u_\lambda^sg_\lambda+h_\lambda$. It remains to verify that the smooth pair $(\hat g_\lambda,f_\lambda)$ satisfies the desired properties. First we point out the facts $g_\lambda=g_{\mathbb E}$ and $f_\lambda=0$ as well as $h_\lambda=0$ outside a compact subset of $M$, where the last one comes from the definition of $\mathcal K_2$. As a consequence, on each end $E_k$ we have
$$
\hat g_\lambda=u_\lambda^sg_{\mathbb E}\mbox{ and }f_\lambda=0\mbox{ when }|x|\geq r_k\mbox{ for some }r_k>0.
$$
Around the infinity, \eqref{Eq: g lambda} becomes $\Delta_{\mathbb{E}}u_\lambda=0$ and it is well-known that $u_\lambda$ has the expression
$$
u_\lambda=1+\frac{A_{\lambda,k}}{r^{n-2}}+O_{2}(r^{1-n}),
$$
where $v=O_{2}(r^{1-n})$ means
\[
|v|+r|\de v|+r^{2}|\de^{2}v|\leq Cr^{1-n}.
\]
As a consequence, the weighted manifold $(M,\hat g_\lambda,f_\lambda)$ is AF with
$$
(\hat{g}_{\lambda})_{ij}-\delta_{ij}, f_{\lambda}\in W^{2,\infty}_{2-n}\mbox{ and } R(\hat{g}_{\lambda}), \Delta_{\hat{g}_{\lambda}} f_{\lambda}\in W^{1,\infty}_{-1-n}.
$$
From Theorem \ref{weighted inequalities}, it is also AF of $(p,\tau)$-type with a possibly smaller $\tau$. To see the small change of mass just notice that
$$
\|P(\hat g_\lambda,f_\lambda)-P(g,f)\|_{L^1_{-2\tau-2}} = \|(1-\eta_\lambda)P(g,f)\|_{L^1_{-2\tau-2}},
$$
which is uniformly bounded since we have
$$
(1-\eta_\lambda)P(g,f)\to 0 \mbox{ in }L^1_{-2-2\tau} \mbox{ as }\lambda\to 0.
$$
By taking $\lambda$ small enough, we can make
$$
\|\hat g-g\|_{W^{2,p}_{-\tau}}\mbox{ and }\|\hat f-f\|_{W^{2,p}_{-\tau}}
$$
arbitrarily small and so Lemma \ref{Lem: continuity mass} yields
$$
|m(M,\hat g_\lambda,f_\lambda,E_k)-m(M,g,f,E_k)|<\ve.
$$
The proof is now completed by taking $(g_\ve,f_\ve)$ to be $(\hat g_\lambda,f_\lambda)$.
\end{proof}

\section{Generalized Geroch conjecture}\label{Sec: Geroch}
\subsection{Generalized Geroch conjecture}
The Geroch conjecture asserts that the $n$-torus $\mathbb{T}^n$ admits no smooth metric with positive scalar curvature. In this subsection, we are going to prove a generalized version for Schoen-Yau-Schick (SYS) manifolds in the setting of $P$-scalar curvature.
Recall
\begin{defn}[SYS manifolds \cite{Gromov2018}]\label{Defn: SYS}
An orientable compact manifold $M^n$ with empty or non-emptry boundary $\partial M$ is said to be SYS if there exist $(n-2)$ cohomology $1$-classes $\beta_1,\ldots,\beta_{n-2}$ in $H^1(M,\partial M,\mathbb Z)$ such that the homology $2$-class
\begin{equation*}
[M,\partial M]\frown(\beta_1\smile\beta_2\smile\cdots\smile \beta_{n-2})\in H_2(M,\mathbb Z)
\end{equation*}
is non-spherical, i.e. any smooth representation cannot consist of $2$-spheres.
\end{defn}

Let $(M,g,f)$ be a weighted manifold with non-empty boundary $\partial M$ and $\nu$ is the outward unit normal of $\partial M$. We say that the boundary is {\it $f$-mean-convex} if $H_g-\mathrm df(\nu)>0$ on $\partial M$, where $H_g$ is the mean curvature of $\partial M$ with respect to $\nu$. The generalized Geroch conjecture is now stated as the following
\begin{prop}\label{Prop: Geroch}
Let $3\leq n\leq 7$. If $M^n$ is a SYS manifold, then it admits no smooth pair $(g,f)$ on $M$ such that the weighted manifold $(M,g,f)$ has positive P-scalar curvature and that the boundary $\partial M$ is $f$-mean-convex (when $\partial M$ is non-empty).
\end{prop}
The proof is inspired from the work \cite{BrendleHirschJohne2022} and let us recall
\begin{defn}(\cite[Definition 1.3]{BrendleHirschJohne2022})\label{Defn: weighted slicing}
Let $(M^n,g,f)$ be a weighted manifold and $m$ be an integer with $1\leq m\leq n-1$. A {\it stable weighted slicing of order $m$} consists of submanifolds $\Sigma_k$ and positive functions $\rho_k\in C^\infty(\Sigma_k)$, $0\leq k\leq m$, satisfying the following conditions:
\begin{enumerate}\setlength{\itemsep}{1mm}
\item [(a)]$\Sigma_0=M$ and $\rho_0=e^{-f}$;
\item [(b)] For each $1\leq k\leq m$, $\Sigma_k$ is an embedded two-sided hypersurface in $\Sigma_{k-1}$, which is a stable critical point of the $\rho_{k-1}$-weighted area given by
$$\mathcal{H}^{n-k}_{\rho_{k-1}}(\Sigma)=\int_{\Sigma} \rho_{k-1} \,\mathrm d\mathcal H^{n-k}_g.$$
\item[(c)] For each $1\leq k\leq m$, the function
$$\frac{\rho_k}{\rho_{k-1}|_{\Sigma_k}}\in C^\infty(\Sigma_k)$$ is a first eigenfunction of the stability operator associated with the $\rho_{k-1}$-weighted area with the first eigenvalue $\lambda_k\geq 0$.
\end{enumerate}
\end{defn}

\begin{rem}
In original \cite[Definition 1.3]{BrendleHirschJohne2022}, Brendle-Hirsch-Johne assume the underlying space to be an (unweighted) Riemannian manifold. While for our purpose, we have to modify the definition to allow the underlying space to be a weighted manifold.
\end{rem}

We have the following lemma from the same argument of \cite[Lemma 3.4]{BrendleHirschJohne2022}, where extra terms $\Delta_{\Sigma_{0}}\log\rho_{0}$ and $H_{\Sigma_{1}}^{2}$ in $\mathcal{E}$ arise due to the possibility $\rho_{0}\neq 1$ here.

\begin{lma}\label{lma:Stable-ineq}
For $m\geq 2$, we have
\begin{equation*}
\int_{\Sigma_m} \rho_{m-1}^{-1} \left(\Lambda+ \mathcal{R}+\mathcal{G}+\mathcal{E}-\Delta_{\Sigma_{0}}\log\rho_{0} \right) \mathrm{dvol}_{m} \leq 0,
\end{equation*}
where
\begin{equation}\label{Eq: quantities}
\left\{
\begin{array}{ll}
\displaystyle \Lambda=\sum_{k=1}^{m-1} \lambda_k;\\[5mm]
\displaystyle \mathcal{R}=\sum_{k=1}^m \Ric_{\Sigma_{k-1}}(\nu_k,\nu_k);\\[5mm]
\displaystyle \mathcal{G}=\sum_{k=1}^{m-1} \left\langle \nabla_{\Sigma_k}\log\rho_k,\nabla_{\Sigma_k}\log\left( \frac{\rho_k}{\rho_{k-1}|_{\Sigma_k}}\right) \right\rangle;\\[5mm]
\displaystyle \mathcal{E}= \sum_{k=1}^m |h_{\Sigma_k}|^2 -\sum_{k=1}^m H_{\Sigma_k}^2.
\end{array}
\right.
\end{equation}
\end{lma}
\begin{proof}
The proof is the same as that of \cite[Lemma 3.4]{BrendleHirschJohne2022} and we omit the details.
\end{proof}

Now we are ready to prove the desired Proposition \ref{Prop: Geroch}.
\begin{proof}[Proof of Proposition \ref{Prop: Geroch}]
First let us show that there always exists a stable weighted slicing of order $(n-1)$ in a SYS weighted manifold $(M^n,g,f)$. We start by letting $\Sigma_0=M$ and $\rho_0=e^{-f}$. From Definition \ref{Defn: SYS}, there are $(n-2)$ cohomology $1$-classes $\beta_1,\ldots,\beta_{n-2}$ in $H^1(M,\partial M,\mathbb Z)$ such that the homology $2$-class
\begin{equation*}
[M,\partial M]\frown(\beta_1\smile\beta_2\smile\cdots\smile \beta_{n-2})\in H_2(M,\mathbb Z)
\end{equation*}
is non-spherical. In particular, we have
$$[M,\partial M]\frown \beta_1\neq 0\in H_{n-1}(M,\mathbb Z).$$
Consider the warped metric $g_{\mathrm{warp}}=g+\rho_0^2\,\mathrm d\theta^2$ on $M\times \mathbb S^1$. We point out the relation
$$\mathcal H^n(\Sigma\times \mathbb S^1,g_{\mathrm{warp}})=\int_\Sigma \rho_0\,\mathrm d\mathcal H^{n-1}_g$$
for any smooth hypersurface $\Sigma\subset M$. As a result, minimizing $\rho_0$-weighted area among hypersurfaces in $M$ is equivalent to minimizing area among $\mathbb S^1$-invariant hypersurfaces in $M\times \mathbb S^1$. Since the mean curvature of $\partial M\times \mathbb S^1$ in $(M\times \mathbb S^1,g_{\mathrm{warp}})$ with respect to outward unit normal satisfies
$$H_{g_{\mathrm{warp}}}=H_g-\mathrm df(\nu)>0,$$
it serves as a barrier for minimizing the area functional.
From the geometric measure theory, we can find a smoothly embedded $\rho_0$-weighted area-minimizing hypersurface $\tilde\Sigma_1$ with multiplicity such that $\tilde \Sigma_1$ is homologous to $[M,\partial M]\frown \beta_1$. Notice that there is at least one component $\Sigma_1$ of $\tilde \Sigma_1$ such that
\begin{equation}\label{Eq: non-spherical 1}
[\Sigma_1]\frown(\beta_2\smile\cdots\smile \beta_{n-2})\in H_2(M,\mathbb Z)
\end{equation}
is non-spherical. Take $\phi_1$ to be a first eigenfunction of the stability operator associated with the $\rho_{0}$-weighted area corresponding to the first eigenvalue $\lambda_1$. Clearly the $\rho_0$-weighted area-minimizing property of $\tilde \Sigma_1$ implies that $\Sigma_1$ is stable as a critical point of the $\rho_{0}$-weighted area and so $\lambda_1\geq 0$. Denote $\rho_1=\rho_0\phi_1$, then
$$
\mathcal S_1=\{(\Sigma_0,\rho_0),(\Sigma_1,\rho_1)\}
$$
is a weighted slicing of order $1$. Notice that the condition \eqref{Eq: non-spherical 1} allows us to extend the weighted slicing above by repeating above construction, and finally we end at a weighted slicing
$$
\mathcal S_{n-2}=\{(\Sigma_0,\rho_0),(\Sigma_1,\rho_1),\ldots,(\Sigma_{n-2},\rho_{n-2})\}
$$
of order $(n-2)$, where $\Sigma_{n-2}$ is an orientable closed surface with positive genus. In particular, we can find an embedded $\rho_{n-2}$-weighted length-minimizing closed geodesic $\Sigma_{n-1}$. As before, we pick up a first eigenfunction $\phi_{n-1}$ of the stability operator associated with the $\rho_{n-2}$-weighted length corresponding to the first eigenvalue $\lambda_{n-1}\geq 0$ and define $\rho_{n-1}=\rho_{n-2}\phi_{n-1}$. Then we find the desired weighted slicing
$$
\mathcal S_{n-1}=\{(\Sigma_0,\rho_0),(\Sigma_1,\rho_1),\ldots,(\Sigma_{n-2},\rho_{n-2}), (\Sigma_{n-1},\rho_{n-1})\}
$$
of order $(n-1)$.

Next we try to deduce a contradiction if there is a smooth pair $(g,f)$ on $M$ such that the weighted manifold $(M,g,f)$ has positive $P$-scalar curvature. Let $\mathcal S_{n-1}$ be the stable weighted slicing constructed above. Then Lemma \ref{lma:Stable-ineq} tells us
\begin{equation}\label{Eq: inequality}
\int_{\Sigma_{n-1}} \rho_{n-2}^{-1} \left(\Lambda+ \mathcal{R}+\mathcal{G}+\mathcal{E}-\Delta_{\Sigma_{0}}\log\rho_{0} \right) \mathrm{dvol}_{n-1} \leq 0,
\end{equation}
where $\Lambda$, $\mathcal{R}$, $\mathcal{G}$ and $\mathcal{E}$ are those quantities given in \eqref{Eq: quantities} with $m=n-1$. The contradiction comes from a delicate analysis on the integrand. For convenience, for any point $x$ in $\Sigma_{n-1}$ we take an orthonormal basis $\{e_{i}\}_{i=1}^{n}$ of $T_xM$ such that $e_i$ is the unit normal of $\Sigma_i$ at $x$ in $\Sigma_{i-1}$ for all $1\leq i\leq n-1$. Following the same computation as in \cite[Lemma 3.8 and Remark 3.9]{BrendleHirschJohne2022} we can obtain
$$
\Lambda+\mathcal{R}+\mathcal{E} \geq  \frac{1}{2}R(g)+\frac{1}{2}\sum_{k=1}^{n-1}\left(|h_{\Sigma_{k}}|^{2}-H_{\Sigma_{k}}^{2}\right)
= \frac{1}{2}R(g)+\frac{1}{2}\sum_{k=1}^{n-2}\left(|h_{\Sigma_{k}}|^{2}-H_{\Sigma_{k}}^{2}\right),
$$
where we used the relation $|h_{\Sigma_{n-1}}|^{2}=H_{\Sigma_{n-1}}^{2}$ since $\Sigma_{n-1}$ is one-dimensional. Notice that we have
\[
|h_{\Sigma_{k}}|^{2} \geq \frac{H_{\Sigma_{k}}^{2}}{n-k}
\]
and
$$H_{\Sigma_{k}}=-\langle\nabla_{\Sigma_{k-1}}\log\rho_{k-1},e_{k}\rangle.$$
It then follows that
$$
\Lambda+\mathcal{R}+\mathcal{E}\geq \frac{1}{2}R(g)-\sum_{k=1}^{n-2}\frac{n-k-1}{2(n-k)}\left|\langle\nabla_{\Sigma_{k-1}}\log\rho_{k-1},e_{k}\rangle\right|^{2}.
$$
Now we take $\mathcal G$ into consideration.
For any $\beta_{k}>0$, it follows from the Cauchy-Schwarz inequality that
\[
\begin{split}
\mathcal{G} = {} & \sum_{k=1}^{n-2} \left\langle \nabla_{\Sigma_k}\log\rho_k,\nabla_{\Sigma_k}\log\left( \frac{\rho_k}{\rho_{k-1}|_{\Sigma_k}}\right) \right\rangle \\
= {} & \sum_{k=1}^{n-2}|\nabla_{\Sigma_k}\log\rho_k|^{2}
-\sum_{k=1}^{n-2}\langle \nabla_{\Sigma_k}\log\rho_k,\nabla_{\Sigma_k}(\log\rho_{k-1}|_{\Sigma_k}) \rangle \\
\geq {} & \sum_{k=1}^{n-2}(1-\beta_{k})\left|\nabla_{\Sigma_k}\log\rho_k\right|^{2}
-\sum_{k=1}^{n-2}\frac{1}{4\beta_{k}}\left|\nabla_{\Sigma_k}(\log\rho_{k-1}|_{\Sigma_k})\right|^{2}.
\end{split}
\]
Collecting all these estimates, we conclude
\[
\begin{split}
& \Lambda+ \mathcal{R}+\mathcal{G}+\mathcal{E}-\Delta_{\Sigma_{0}}\log\rho_{0} \\
\geq {} & \frac{1}{2}R(g)-\Delta_{\Sigma_{0}}\log\rho_{0}-\sum_{k=1}^{n-2}\frac{n-k-1}{2(n-k)}\,\left|\langle\nabla_{\Sigma_{k-1}}\log\rho_{k-1},e_{k}\rangle\right|^{2} \\
& +\sum_{k=1}^{n-2}(1-\beta_{k})|\nabla_{\Sigma_k}\log\rho_k|^{2}
-\sum_{k=1}^{n-2}\frac{1}{4\beta_{k}}|\nabla_{\Sigma_k}(\log\rho_{k-1}|_{\Sigma_k})|^{2}.
\end{split}
\]
Taking
\[
\beta_{k} = \frac{n-k}{2(n-k-1)}
\]
and using the relation
\[
\left|\langle\nabla_{\Sigma_{k-1}}\log\rho_{k-1},e_{k}\rangle\right|^{2}+|\nabla_{\Sigma_k}(\log\rho_{k-1}|_{\Sigma_k})|^{2} = |\nabla_{\Sigma_k-1}\log\rho_{k-1}|^{2},
\]
we obtain
\[
\begin{split}
& \Lambda+ \mathcal{R}+\mathcal{G}+\mathcal{E}-\Delta_{\Sigma_{0}}\log\rho_{0} \\
\geq {} & \frac{1}{2}R(g)-\Delta_{\Sigma_{0}}\log\rho_{0}
-\sum_{k=1}^{n-2}\frac{1}{4\beta_k}\,\left|\nabla_{\Sigma_{k-1}}\log\rho_{k-1}\right|^{2} +\sum_{k=1}^{n-2}(1-\beta_{k})|\nabla_{\Sigma_k}\log\rho_k|^{2}.
\end{split}
\]
Since we have
$$
1-\beta_k=\frac{1}{4\beta_{k+1}}\mbox{ for }k=1,2,\ldots, n-3,
$$
after canceling terms in last two sums, we see that
\begin{equation}\label{Eq: key inequality}
\begin{split}
& \Lambda+ \mathcal{R}+\mathcal{G}+\mathcal{E}-\Delta_{\Sigma_{0}}\log\rho_{0} \\
\geq {} & \frac{1}{2}R(g)-\Delta_{\Sigma_{0}}\log\rho_{0}-\frac{n-2}{2(n-1)}\left|\nabla_{\Sigma_{0}}\log\rho_{0}\right|^{2}\geq \frac{1}{2}P(g,f)>0.
\end{split}
\end{equation}
This contradicts to \eqref{Eq: inequality} and we complete the proof.
\end{proof}
\subsection{From quasi-positivity to global positivity} In the following, we say that a weighted manifold $(M,g,f)$ has {\it quasi-positive} $P$-scalar curvature if the $P$-scalar curvature is non-negative everywhere and is positive at some point. We establish the following deformation result to extend the positivity of $P$-scalar curvature from one point to the entire manifold.

\begin{prop}\label{Prop: spread}
If $(M,g, f)$ is a closed weighted manifold with quasi-positive $P$-scalar curvature and $f$-mean-convex boundary, then there is a new smooth metric $\tilde g$ on $M$ such that the weighted manifold $(M,\tilde g,f)$ has positive $P$-scalar curvature and $f$-mean-convex boundary.
\end{prop}
\begin{proof}
Let us consider the operator
$$
\tilde T_{g,f}:C^{2,\alpha}(M)\to C^\alpha(M)\times C^{1,\alpha}(\partial M),\,u\mapsto \left(u^{s+1} \cdot T_{g,f}(u),\frac{\partial u}{\partial \nu}\right),
$$
where $T_{g,f}$ is the operator given by $T_{g,f}(u)=P(u^s g,f)$. From Lemma \ref{Lem: linearization}, we can compute
\[
D\tilde T_{g,f}|_{u=1}(v)=\left(-(n-1)s\Delta_g v +4\langle\mathrm{d}v,\mathrm{d}f\rangle_g+P(g,f)v,\frac{\partial v}{\partial\nu}\right).
\]

Let us denote $\tilde A=D\tilde T_{g,f}|_{u=1}$ for short and show that the operator $\tilde A$ is surjective. From the Fredholm alternative (refer to \cite[Theorem 2.29]{Lieberman2013}), we just need to show $\mathrm{Ker}\tilde A=0$. To see this, we take a function $\phi\in C^{2,\alpha}(M)$ such that
$\tilde A(\phi)=(0,0).$ It follows from the strong maximum principle that $\phi$ is a constant function. Combining with the fact that $P(g,f)$ is quasi-positive, we see that $\phi$ must be identical to zero and so $\tilde A$ is surjective.

It follows that we are able to take a function $v_0\in C^{2,\alpha}(M)$ such that
$\tilde A(v_0)=(1,0)$. Since $(g,f)$ is a smooth pair, we can conclude the smoothness of $v_0$ from the Schauder estimates. Notice that we have
$$
\left.\frac{\mathrm d}{\mathrm dt}\right|_{t=0}(1+tv_0)^{s+1}P(g_t,f)=1\mbox{ with }g_t=(1+tv_0)^sg.$$
This yields $P(g_t,f)>0$ for sufficiently small $t$. Using $\frac{\de v_{0}}{\de\nu}=0$, it is also not difficult to verify
$$H_{g_t}-\mathrm df(\nu_t)=(1+tv_0)^{-\frac{s}{2}}(H_g-\mathrm df(\nu))>0.$$
So we complete the proof.
\end{proof}
Combined with Proposition \ref{Prop: Geroch} we have the following
\begin{cor}\label{Cor: Geroch}
Let $3\leq n\leq 7$. If $M^n$ is a SYS manifold, then it admits no smooth pair $(g,f)$ on $M$ such that the weighted manifold $(M,g,f)$ has quasi-positive P-scalar curvature and that the boundary $\partial M$ is $f$-mean-convex (when $\partial M$ is non-empty).
\end{cor}

\section{Proof of Theorem \ref{Thm: main}}\label{Sec: Proof}
\subsection{The mass inequality}
We are ready to prove the mass inequality \eqref{Eq: mass inequality}.
\begin{proof}[Proof of the mass inequality \eqref{Eq: mass inequality}]
We argue by contradiction. Suppose that $(M^n,g,f)$ is an AF weighted manifolds of $(p,\tau)$-type with non-negative $P$-scalar curvature, which has negative mass on some end $E_-$.
Fix a positive $\ve$ with
$$\ve<|m(M,g,f,E_-)|.$$
From Proposition \ref{Density}, we can find a new smooth pair $(g_\ve,f_\ve)$ such that
\begin{itemize}\setlength{\itemsep}{1mm}
    \item $(M,g_\ve,f_\ve)$ is an AF weighted manifold of $(p,\tau)$-type;
    \item the mass of the end $E_-$ satisfies $m(M,g_\ve,f_\ve,E_-)<0$;
    \item for end $E_-$ there is a constant $r_-$ such that
$$
g_\ve=u_\ve^s g_{\mathbb E}\mbox{ and }f_\ve\equiv 0\mbox{ when }|x|\geq r_-,
$$
where $u_\ve$ is a harmonic function with the expression
$$
u_\ve=1+\frac{A_\ve}{r^{n-2}}+O_{2}(r^{1-n}).
$$
\end{itemize}
From \eqref{equivalent formula weighted mass} and a direct computation, we have
\[
\begin{split}
m(M,g_\ve,f_\ve, E_-)
= {} & \lim_{\rho\to \infty} \int_{S_\rho} \big(\partial_j(g_{\ve})_{ij}-\partial_i(g_{\ve})_{jj}+2\partial_i f_{\ve}\big) \,\nu^i \, \mathrm{d}\sigma \\
= {} & \lim_{\rho\to \infty} \int_{S_\rho}(\partial_iu_\ve^s-n\partial_iu_\ve^s)\,\nu^i \, \mathrm{d}\sigma \\[2mm]
= {} & (1-n)(2-n)\omega_{n-1} s A_\ve,
\end{split}
\]
which implies
$$A_\ve=\frac{m(M,g_\ve,f_\ve,E_-)}{4(n-1)\omega_{n-1}}<0.$$
Applying the Lohkamp compactification trick from \cite{Lohkamp1999}, we can close the end $E_-$ by $\mathbb T^n$. Namely, we can deform the metric $g_\ve$ to be Euclidean around the infinity of $E_-$ such that the $P$-scalar curvature becomes quasi-positive. Then we take a large cube in the end $E_-$ and glue opposite faces. For those ends other than $E_-$, we are able to cut away them along large mean-convex coordinate spheres outside the support of $f_\ve$.

Now we obtain a compact weighted manifold $(\hat M,\hat g,\hat f)$ with quasi-positive $P$-scalar curvature and $\hat f$-mean-convex boundary, where $\hat M$ has the form $\mathbb T^n\# N$ for some compact manifold $N$. It follows from Corollary \ref{Cor: Geroch} that we can deduce a contradiction by showing
\begin{claim}
    $\hat M=\mathbb T^n\#N$ is a SYS manifold.
\end{claim}
Clearly there is a degree-one map $F:(\hat M,\partial\hat M)\to (\mathbb T^n,\mathrm{pt})$. For the $n$-torus $\mathbb T^n$, we can find $(n-2)$ cohomology $1$-classes $\beta_1,\beta_2\ldots,\beta_{n-2}$ such that
$$[\mathbb T^n]\frown(\beta_1\smile\beta_2\smile\cdots \smile \beta_{n-2})=[\mathbb T^2].$$
Consider the exact sequence
$$H^0(\mathbb T^n,\mathbb Z)\to H^0(\mathrm{pt},\mathbb Z)\to H^1(\mathbb T^n,\mathrm{pt},\mathbb Z)\to H^1(\mathbb T^n,\mathbb Z)\to H^1(\mathrm{pt},\mathbb Z).$$
Since the map $H^0(\mathbb T^n,\mathbb Z)\to H^0(\mathrm{pt},\mathbb Z)$ is surjective and $H^1(\mathrm{pt},\mathbb Z)=0$, we see that $H^1(\mathbb T^n,\mathrm{pt},\mathbb Z)$ is isomorphic to $H^1(\mathbb T^n,\mathbb Z)$. In particular, we can find cohomology $1$-classes $\bar\beta_1,\bar\beta_2,\ldots,\bar\beta_{n-2}$ in $H^1(\mathbb T^n,\mathrm{pt},\mathbb Z)$ such that $\id^*\bar\beta_i=\beta_i$.

Denote $\hat\beta_i=F^*\bar\beta_i$. Let us show that the homology $2$-class
$$\sigma=[\hat M,\partial\hat M]\frown(\hat\beta_1\smile\hat\beta_2\smile\cdots\smile\hat\beta_{n-2})$$
is non-spherical. To see this, we compute
\[
\begin{split}
F_*\sigma
= {} & (\deg F)\cdot[\mathbb T^n,\mathrm{pt}]\frown(\bar\beta_1\smile\bar\beta_2\smile\cdots\smile\bar\beta_{n-2})\\
= {} & (\deg F)\cdot\id_*\left([\mathbb T^n]\frown(\beta_1\smile\beta_2\smile\cdots\smile\beta_{n-2})\right)\\
= {} & (\deg F)\cdot[\mathbb T^2]=[\mathbb T^2].
\end{split}
\]
Therefore, any smooth representation $\Sigma$ of $\sigma$ admits a degree-one map $\pi\circ F:\Sigma \to \mathbb T^2$, where $\pi:\mathbb T^n\to\mathbb T^2$ is the projection map. This guarantees that $\Sigma$ cannot consist of $2$-spheres.
\end{proof}

\subsection{Rigidity} In this subsection, we assume that there is an end $E_0$ with $$m(M,g,f,E_0)=0.$$
In the followiing, we show that $(M,g,f)$ is isometric to $(\mathbb R^n,g_{\mathbb E},0)$.
\begin{proof}[Proof of the rigidity]
The proof will be divided into two steps:

\bigskip
\noindent
{\bf Step 1.} $P(g,f)$ vanishes everywhere.
\bigskip

Otherwise we can find a point $x_0\in M$ such that $P(g,f)(x_0)>0$. Our idea is using the conformal deformation to decrease the positive $P$-scalar curvature such that the mass of the end $E_0$ becomes negative. Then we can deduce a contradiction from the mass inequality \eqref{Eq: mass inequality}.

Given a smooth positive function $u$ on $M$, it is direct to compute
$$
P(u^sg,f)=u^{-(1+s)}\big(-(n-1)s\Delta_g u+4\langle\mathrm du,\mathrm df\rangle_g+P(g,f)u\big).
$$
Consider the operator
$$
\hat T_{g,f}:W^{2,p}_{-\tau}\to L^p_{-\tau-2},\,v\mapsto -(n-1)s\Delta_g v+4\langle\mathrm dv,\mathrm df\rangle_g+P(g,f)v.
$$
From Theorem \ref{Fredholm}, we know that $\hat T_{g,f}$ is a Fredholm operator. Since the Fredholm index is homotopy invariant, we conclude that the Fredholm index of $\hat T_{g,f}$ equals to the Fredholm index of $\Delta_g$, which is zero since $\tau\in(\frac{n-2}{2},n-2)$ (see \cite[Proposition 2.2]{Bartnik1986}). We point out that the operator $\hat T_{g,f}$ actually has a trivial kernel. To see this, we take a function $v_0\in W^{2,p}_{-\tau}$ such that $\hat T_{g,f}(v_0)=0$. It follows from the Sobolev embedding and Schauder estimates that $v_0$ is a smooth function. From the strong maximum principle and the quasi-positivity of $P(g,f)$, we conclude $v_0\equiv 0$. This guarantees the existence of a solution $v_\eta\in W^{2,p}_{-\tau}$ of the equation
$$
\hat T_{g,f}v_\eta=-\eta P(g,f),
$$
where $\eta:M\to[0,1]$ is a prescribed smooth non-negative cut-off function with $\eta(x_0)>0$ such that $\supp\,\eta$ is compactly supported in $\{P(g,f)>0\}$. Since $v_\eta$ approaches zero at the infinity of each end, we have $-1<v_\eta<0$ from the comparison principle and in particular $u_\eta:=1+v_\eta$ is positive. It follows from the proof of \cite[Theorem 1.17]{Bartnik1986} that on end $E_0$ we have
\begin{equation}\label{Eq: expansion}
v_\eta=c_\eta |x|^{2-n}+w_\eta\mbox{ with }w_\eta\in W^{2,p}_{2-n-\theta} \mbox{ for any } \theta\in(0,1).
\end{equation}
From the Sobolev embedding, we see $w_\eta \in W^{1,\infty}_{2-n-\theta}$. Let
$$g_\eta=u_\eta^s g$$ and as in the proof of Proposition \ref{Density} it is not difficult to verify that $(M,g_\eta,f)$ is still an AF weighted manifold of $(p,\tau)$-type with a possibly smaller $\tau$, which has non-negative $P$-scalar curvature.

Now we want to compute the mass change under the conformal deformation. Notice that we have
\[
\begin{split}
&\int_{S_\rho} \big (\partial_j(u_\eta^sg_{ij})-\partial_i(u_\eta^sg_{jj})+2\partial_if\big)\, \nu^i \, \mathrm{d}\sigma\\
=&\int_{S_\rho} \big (u_\eta^s(\partial_jg_{ij}-\partial_ig_{jj})+2\partial_if\big) \,\nu^i \, \mathrm{d}\sigma+(1-n)\int_{S_\rho}(\partial_i u_{\eta}^s)\,\nu^i\, \mathrm{d}\sigma\\
=&\int_{S_\rho} \big (\partial_jg_{ij}-\partial_ig_{jj}+2\partial_if\big) \,\nu^i \, \mathrm{d}\sigma+O(\rho^{-\tau})+(1-n)\int_{S_\rho}(\partial_i u_{\eta}^s)\,\nu^i\,\mathrm{d}\sigma
\end{split}
\]
and
\[
\begin{split}
\int_{S_\rho}(\partial_i u_{\eta}^s)\, \nu^i\, \mathrm{d}\sigma&=\int_{S_\rho} su_\eta^{s-1}\big((2-n)c_\eta\rho^{1-n}+(\partial_iw_\eta)\nu^i\big)\,\mathrm{d}\sigma\\
&=-4\omega_{n-1}\cdot c_\eta+O(\rho^{2-n})+O(\rho^{-\theta}).
\end{split}
\]
By \eqref{equivalent formula weighted mass}, we have
$$
m(M,g_\eta,f,E_0)=m(M,g,f,E_0)+4(n-1)\omega_{n-1}c_\eta.
$$

From the mass inequality \eqref{Eq: mass inequality}, it remains to show $c_\eta<0$ for a contradiction. Let us do this by constructing a barrier function for comparison. We point out that there is a $r_0>0$ such that $\hat T_{g,f}v_\eta=0$ when $|x|\geq r_0$. Let
$$\bar v=|x|^{2-n-\frac{\theta}{2}}.$$
A direct computation shows
$$
\hat T_{g,f}\bar v=\frac{2\theta(n-1)}{n-2}\left(2-n-\frac{\theta}{2}\right)|x|^{-n-\frac{\theta}{2}}+O\left(|x|^{-n-\frac{\theta}{2}-\tau}\right)+O\left(|x|^{-n-\frac{\theta}{2}-2\tau}\right).
$$
After possibly increasing the value of $r_0$ we obtain $\hat T_{g,f}\bar v\leq 0$ when $|x|\geq r_0$. From the comparison principle, it is easy to see
$$
v_{\eta}(x)\leq \left(\sup_{S_{r_0}}v_{\eta}\right)\left(\frac{|x|}{r_0}\right)^{2-n-\frac{\theta}{2}} \mbox{ for all }|x|\geq r_0.
$$
Combining this with \eqref{Eq: expansion}, we have
$$
c_\eta\leq \left(\sup_{S_{r_0}}v_{\eta}\right) r_0^{n-2+\frac{\theta}{2}}|x|^{-\frac{\theta}{2}}+O\big(|x|^{-\theta}\big).
$$
By taking $|x|$ large enough, we obtain $c_\eta<0$ as desired.

\bigskip
\noindent
{\bf Step 2.} $(M,g,f)$ is isometric $(\mathbb R^n,g_{\mathbb E},0)$.
\bigskip

Let us show $f\equiv 0$, whose argument depends on an extension of the vanishing property of $P$-scalar curvature to $P_n$-scalar curvature. For a weighted manifold $(M,g,f)$ we introduce the $P_n$-scalar curvature given by
$$
P_n(g,f)=R(g)+2\Delta_g f-\frac{n-2}{n-1}|\mathrm df|^2_g.
$$
The argument from Section \ref{Sec: Density} and \ref{Sec: Geroch} actually yields that if $(M^n,g,f)$ is an AF weighted manifold with non-negative $P_n$-scalar curvature, then for each end $E$ we have $m(M,g,f,E)\geq 0$. To see this one just need to notice
\begin{itemize}\setlength{\itemsep}{1mm}
\item the reduction from the mass inequality to the generalized Geroch conjecture does not depend on the prescribed coefficient of the term $|\mathrm df|^2_g$;
\item the key for the validity of the generalized Geroch conjecture is the inequality \eqref{Eq: key inequality}, where we have the requirement that the prescribed coefficient of the term $|\mathrm df|^2_g$ is no less than $\frac{n-2}{n-1}$.
\end{itemize}
Combining the mass inequality and the argument in Step 1, we are able to prove $P_n(g,f)\equiv 0$. As a result, we obtain
$$
|\mathrm df|_g^2=(n-1)\big(P_n(g,f)-P(g,f)\big)=0\mbox{ and so }f\equiv 0.
$$

Now we see that $(M,g)$ is an asymptotically flat manifold of Sobolev type $(p,\tau)$ defined in \cite{LeeLesourdUnger2022} with non-negative scalar curvature and zero mass on $E_0$. It follows from \cite[Theorem 1.5]{LeeLesourdUnger2022} that $(M,g)$ is isometric to the Euclidean space and we complete the proof for rigidity.
\end{proof}

\appendix
\section{Weighted analysis}\label{Sec: Appendix A}
In this appendix, we include some basic notions and results on weighted spaces and weighted analysis.

\subsection{Weighted Sobolev spaces}
Let $B_{1}\subset\mathbb{R}^{n}$ be the unit ball. The weighted Sobolev space on $\mathbb R^n\setminus B_1$ is defined as follows.

\begin{defn}
\label{weighted Sobolev space}
For $k\in\mathbb{Z}_{\geq0}$, $p\geq1$ and $\tau\in\mathbb{R}$, the weighted Sobolev space on $\mathbb R^n\setminus B_1$ is defined to be
\[
W_{-\tau}^{k,p}(\mathbb{R}^{n}\setminus B_{1}) := \left\{u\in W_{\mathrm{loc}}^{k,p}(\mathbb{R}^{n}\setminus B_{1}):
\|u\|_{W_{-\tau}^{k,p}(\mathbb{R}^{n}\setminus B_{1})} < \infty \right\},
\]
where
\[
\|u\|_{W_{-\tau}^{k,p}(\mathbb{R}^{n}\setminus B_{1})} := \sum_{0\leq |I|\leq k}
\left(\int_{\mathbb{R}^{n}\setminus B_{1}}\left(|\partial^{I}u|\cdot|x|^{|I|+\tau}\right)^{p}|x|^{-n}\,\mathrm dx\right)^{\frac{1}{p}}.
\]
When $k=0$, we usually use $L_{-\tau}^{p}(\mathbb R^n\setminus B_1)$ instead of $W_{-\tau}^{0,p}(\mathbb R^n\setminus B_1)$.
\end{defn}

Let $M$ be a differentiable manifold such that there is a compact subset $K\subset M$ such that the complement $M\setminus K$ consists of finitely many ends $E_1, E_2,\ldots, E_N$ and each end $E_k$ admits a diffeomorohism $\Phi_k:E_k\to \mathbb{R}^{n}\setminus B_{1}$. From the compactness of $K$, we also fix finitely many coordinate charts $\{U_i\}_{i=1}^{N'}$ to cover $K$. With respect to these coordinate charts, we can define weighted Sobolev spaces on $M$ as follows.
\begin{defn}
The weighted Sobolev space $W^{k,p}_{-\tau}$ on $M$ is defined to be
\[
W_{-\tau}^{k,p}:= \left\{ u\in W_{\mathrm{loc}}^{k,p}(M):
\sum_{k=1}^N\|u\|_{W_{-\tau}^{k,p}(E_k)}+\sum_{i=1}^{N'}\|u\|_{W^{k,p}(U_i)} < \infty \right\}.
\]
For any $u$ in $W^{k,p}_{\tau}$, we denote
$$
\|u\|_{W^{k,p}_{-\tau}}=\sum_{k=1}^N\|u\|_{W_{-\tau}^{k,p}(E_k)}+\sum_{i=1}^{N'}\|u\|_{W^{k,p}(U_i)}.
$$
\end{defn}

\begin{thm}[Theorem 1.2 of \cite{Bartnik1986}]\label{weighted inequalities}
We have the following inequalities.
\begin{enumerate}
\item[(i)] If $1\leq p\leq q\leq\infty$, $\tau_{1}<\tau_{2}$ and $u\in L_{-\tau_{2}}^{q}$, then $u\in L_{-\tau_{1}}^{p}$ and
\[
\|u\|_{L_{-\tau_{1}}^{p}} \leq C\|u\|_{L_{-\tau_{2}}^{q}}.
\]

\item[(ii)] If $u\in L_{-\tau_{1}}^{p}$, $v\in L_{-\tau_{2}}^{q}$, $\tau=\tau_{1}+\tau_{2}$ and $\frac{1}{r}=\frac{1}{p}+\frac{1}{q}$, then
\[
\|uv\|_{L_{-\tau}^{r}} \leq \|u\|_{L_{-\tau_{1}}^{p}} \|v\|_{L_{-\tau_{2}}^{q}}.
\]

\item[(iii)] For $u\in W_{-\tau}^{k,p}$, if $n-kp>0$, then $u\in L_{-\tau}^{\frac{np}{n-kp}}$ and
\[
\|u\|_{L_{-\tau}^{\frac{np}{n-kp}}} \leq C\|u\|_{W_{-\tau}^{k,p}}.
\]
If $n-kp<0$, then $u\in L_{-\tau}^{\infty}$ and
\[
\|u\|_{L_{-\tau}^{\infty}} \leq C\|u\|_{W_{-\tau}^{k,p}}.
\]
\end{enumerate}
\end{thm}

\subsection{Weighted analysis on
AF weighted manifold}In the following, we take $(M,g,f)$ to be an AF weighted manifold of $(p,\tau)$-type.  Below we collect some basic theorems for weighted analysis.

\begin{thm}\label{elliptic theory}
Assume that $X$ is a smooth vector field in $L^{\infty}_{-\tau-1}$ and that $c$ is a smooth function in $L^\infty_{-\tau-2}$. For any $q>1$ and $\sigma\in\mathbb{R}$, if $u\in L^q_{-\sigma}$ satisfies the equation
$$
\Delta_g u+\langle\mathrm du,X\rangle_g+cu=h \mbox{ with }h\in L^{q}_{-\sigma-2}
$$
in the distributional sense with respect to $\mathrm{dvol}_{g}$, i.e.
$$
\int_{M}\left(\Delta_g\vp-\langle\mathrm d\vp,X\rangle_g-\vp\,\Div_g X+c\vp\right)u\,\mathrm{dvol}_{g} =\int_M \vp h\,\mathrm{dvol}_{g}
$$
holds for all $\vp\in C^\infty_0(M)$,
then $u$ is in $W^{2,q}_{-\sigma}$ and we have
\begin{equation}\label{Eq: Lq estimate}
\|u\|_{W^{2,q}_{-\sigma}}\leq C\left(\|h\|_{L^q_{-\sigma-2}}+\|u\|_{L^{q}_{-\sigma}}\right)
\end{equation}
for some constant $C$ depending only on $\|g-g_{\mathbb{E}}\|_{L_{-\tau}^{\infty}}$, $\|\de g\|_{L_{-\tau-1}^{\infty}}$, $\|X\|_{L_{-\tau-1}^{\infty}}$, $\|c\|_{L_{-\tau-2}^{\infty}}$, $M$, $q$, $\sigma$ and $\tau$.
\end{thm}
\begin{proof}
First we show $u\in W^{2,q}_{\mathrm{loc}}$. It suffices to show that $u\in W_{\mathrm{loc}}^{2,q}(B)$ for any ball $B\subset M$. By \cite[Theorem 9.15]{GilbargTrudinger2001} and $h-cu\in L^{q}(B)$, there exists $v\in W^{2,q}(B)$ solving the Dirichlet problem
\[
\begin{cases}
\ \Delta_g v+\langle\mathrm dv,X\rangle_g = f-cu & \mbox{in $B$,} \\[0.5mm]
\ v = 0 & \mbox{on $\de B$.}
\end{cases}
\]
Write $w=u-v$. Then for all $\vp\in C^\infty_0(B)$,
$$
\int_{M}\left(\Delta_g\vp-\langle\mathrm d\vp,X\rangle_g-\vp\,\Div_g X\right)w\,\mathrm{dvol}_g = 0.
$$
Since $\mathrm{dvol}_{g}=\sqrt{\det g}\,\mathrm{d}x$, then
$$
\int_{M}\left(\sqrt{\det g}\,\Delta_g\vp-\langle\mathrm d\vp,\sqrt{\det g}\,X\rangle_g-\vp\sqrt{\det g}\,\Div_g X\right)w\,\mathrm{d}x = 0.
$$
Define the elliptic operator $\hat{L}$ by
\[
\hat{L}\vp = \sqrt{\det g}\,\Delta_g\vp-\langle\mathrm d\vp,\sqrt{\det g}\,X\rangle_g+\vp\sqrt{\det g}\,\Div_g X,
\]
and let $\tilde{L}$ be the $L^2$-adjoint of $\hat{L}$ with respect to $\mathrm{d}x$. Then the above shows that $\tilde{L}w=0$ in $B$ in the distributional sense with respect to $\mathrm{d}x$. Using \cite[Theorem 6.33]{Folland1995}, we obtain $w \in H_{s}^{\mathrm{loc}}(B)$ for any $s$. It follows that $w\in C_{\mathrm{loc}}^{\infty}(B)$ and so $u=w+v\in W^{2,q}_{\mathrm{loc}}(B)$.

Now the desired estimate \eqref{Eq: Lq estimate} follows from the weighted $L^q$-estimate \cite[Proposition 1.6]{Bartnik1986} applied to each end $E_k$ and the $L^q$-estimate \cite[Theorem 9.11]{GilbargTrudinger2001} applied to each $U_i$.
We notice that the requirement $q\leq p$ in \cite[Proposition 1.6]{Bartnik1986} is unnecessary in our case and we include the details in Lemma \ref{weighted Lq} for completeness.
\end{proof}

\begin{lma}\label{weighted Lq}
Suppose that the operator
\[
L = a^{ij}\de_{i}\de_{j}+b^{i}\de_{i}+c
\]
satisfying
\begin{itemize}\setlength{\itemsep}{1mm}
\item $\Lambda^{-1}\delta_{ij}\leq a^{ij}\leq\Lambda\delta_{ij}$;
\item $\|a^{ij}\|_{W_{-\tau}^{1,\infty}(\mathbb R^n)}+\|b^{i}\|_{L_{-\tau-1}^{\infty}(\mathbb R^n)}+\|c\|_{L_{-\tau-2}^{\infty}(\mathbb R^n)}\leq\Lambda$
\end{itemize}
for some $\Lambda>1$. If $u\in L_{-\sigma}^{q}(\mathbb R^n)\cap W_{\mathrm{loc}}^{2,q}(\mathbb R^n)$ and $Lu\in L_{-\sigma-2}^{q}(\mathbb R^n)$ for some $q>1$ and $\sigma\in\mathbb{R}$, then $u\in W_{-\sigma}^{2,q}(\mathbb R^n)$ and
\begin{equation}\label{Lp estimate eqn 1}
\|u\|_{W_{-\sigma}^{2,q}(\mathbb R^n)} \leq C\|Lu\|_{L_{-\sigma-2}^{q}(\mathbb R^n)}+C\|u\|_{L_{-\sigma}^{q}(\mathbb R^n)},
\end{equation}
where $C$ is a constant depeding only on $n$, $q$, $\tau$, $\sigma$ and $\Lambda$.
\end{lma}

\begin{proof}
Write $h=Lu$ and $A_{r_{1},r_{2}}=\{r_{1}<|x|<r_{2}\}$ and define
\[
a_{R}^{ij}(x) = a^{ij}(Rx), \ \
b_{R}^{i}(x) = Rb^{i}(Rx), \ \
c_{R}(x) = R^{2}c(Rx),
\]
\[
L_{R} = a_{R}^{ij}\de_{i}\de_{j}+b_{R}^{i}\de_{i}+c_{R}, \ \
u_{R}(x) = u(Rx), \ \
h_{R}(x) = R^{2}h(Rx).
\]
It follows that $L_{R}u_{R}=h_{R}$. Applying the $L^{q}$-estimate \cite[Theorem 9.11]{GilbargTrudinger2001} to $L_{R}u_{R}=h_{R}$ on $A_{2,4}\subset A_{1,8}$, we see that
\[
\|u_{R}\|_{W^{2,q}(A_{2,4})} \leq C\|h_{R}\|_{L^{q}(A_{1,8})}+C\|u_{R}\|_{L^{q}(A_{1,8})},
\]
which is equivalent to
\[
\int_{A_{2,4}}(|u_{R}|^{q}+|\de u_{R}|^{q}+|\de^{2}u_{R}|^{q})\,\mathrm{d}x \leq C\int_{A_{1,8}}(|h_{R}|^{q}+|u_{R}|^{q})\,\mathrm{d}x.
\]
Then we have
\[
\int_{A_{2R,4R}}\Big(|u|^{q}+(R|\de u|)^{q}+(R^{2}|\de^{2}u_{R}|)^{q}\Big)R^{-n}\,\mathrm{d}x \leq C\int_{A_{R,8R}}\Big((R^{2}|h|)^{q}+|u|^{q}\Big)R^{-n}\,\mathrm{d}x
\]
and so
\[
\begin{split}
\int_{A_{2R,4R}}\Big((R^{\sigma}|u|)^{q}+(R^{\sigma+1}|\de u|)^{q}&+(R^{\sigma+2}|\de^{2}u_{R}|)^{q}\Big)R^{-n}\,\mathrm{d}x\\
&
\leq C\int_{A_{R,8R}}\Big((R^{\sigma+2}|h|)^{q}+(R^{\sigma}|u|)^{q}\Big)R^{-n}\,\mathrm{d}x,
\end{split}
\]
which implies
\[
\|u\|_{W_{-\sigma}^{2,q}(A_{2R,4R})}^{q} \leq C\|Lu\|_{L_{-\sigma-2}^{q}(A_{R,8R})}^{q}+C\|u\|_{L_{-\sigma}^{q}(A_{R,8R})}^{q}.
\]
It is clear that
\begin{equation}\label{Lp estimate eqn 2}
\begin{split}
& \|u\|_{W_{-\sigma}^{2,q}(A_{2,\infty})}^{q} = \sum_{k=1}^{\infty}\|u\|_{W_{-\sigma}^{2,p}(A_{2^{k},2^{k+1}})}^{q} \\
\leq {} & C\sum_{k=1}^{\infty}\|Lu\|_{L_{-\sigma-2}^{q}(A_{2^{k-1},2^{k+1}})}^{q}+C\sum_{k=1}^{\infty}\|u\|_{L_{-\sigma}^{q}(A_{2^{k-1},2^{k+1}})}^{q} \\[2mm]
\leq {} & C\|Lu\|_{L_{-\sigma-2}^{q}(A_{1,\infty})}^{q}+C\|u\|_{L_{-\sigma}^{q}(A_{1,\infty})}^{q}.
\end{split}
\end{equation}
Applying the $L^{q}$-estimate \cite[Theorem 9.11]{GilbargTrudinger2001} again,
\begin{equation}\label{Lp estimate eqn 3}
\|u\|_{W_{-\tau}^{2,q}(B_{4})}^{q}
\leq C\|Lu\|_{L_{-\sigma-2}^{q}(B_{8})}^{q}+C\|u\|_{L_{-\sigma}^{q}(B_{8})}^{q}.
\end{equation}
Then \eqref{Lp estimate eqn 1} follows from \eqref{Lp estimate eqn 2} and \eqref{Lp estimate eqn 3}.
\end{proof}

\begin{thm}\label{Fredholm}
Assume that $X$ is a smooth vector field in $L^{\infty}_{-\tau-1}$ with $\Div_{g}X\in L^{\infty}_{-\tau-2}$ and that $c$ is a smooth function in $L^\infty_{-\tau-2}$. Then the operator
$$
L:W^{2,q}_{-\tau}\to L^q_{-\tau-2},\ u\mapsto  \Delta_g u+\langle\mathrm du,X\rangle_g+cu,
$$
is Fredholm.
\end{thm}
\begin{proof}
Let us split the argument into the following four steps.

\bigskip
\noindent
{\bf Step 1.} $\|u\|_{W_{-\tau}^{2,q}} \leq C\|Lu\|_{L_{-\tau-2}^{q}}+C\|u\|_{L^{q}(\Omega)}$ for any $u\in W_{-\tau}^{2,q}$, where $\Omega\subset M$ is a bounded subset and $C$ is a constant depending only on $\|g-g_{\mathbb{E}}\|_{L_{-\tau}^{\infty}}$, $\|\de g\|_{L_{-\tau-1}^{\infty}}$, $\|X\|_{L_{-\tau-1}^{\infty}}$, $\|c\|_{L_{-\tau-2}^{\infty}}$, $M$, $q$ and $\tau$.

\bigskip

Denote the geodesic ball in $M$ by $B_{R}$ and let $\eta_{R}$ be a smooth cut-off function on $M$ such that
\[
\text{$\eta_{R} \equiv 1$ in $B_{R}$ and $\eta_{R} \equiv 0$ in $M\setminus B_{2R}$},
\]
where $R$ is a sufficiently large constant to be determined later. Write $u_{R}=\eta_{R}u$. Applying \cite[(1.22)]{Bartnik1986} to $(u-u_{R})$ in each end $E_{k}$, we have
\begin{equation}\label{Fredholm eqn 5}
\begin{split}
& \|u-u_{R}\|_{W_{-\tau}^{2,q}(M\setminus B_{R})}
\leq C\|\Delta_{\mathbb{E}}(u-u_{R})\|_{L_{-\tau-2}^{q}(M\setminus B_{R})} \\[1mm]
\leq {} & C\|L(u-u_{R})\|_{L_{-\tau-2}^{q}(M\setminus B_{R})}+C\|(L-\Delta_{\mathbb{E}})(u-u_{R})\|_{L_{-\tau-2}^{q}(M\setminus B_{R})}.
\end{split}
\end{equation}
For the second term, since $g-g_{\mathbb{E}}\in L_{-\tau}^{\infty}$, $\de g\in L_{-\tau-1}^{\infty}$, $X\in L_{-\tau-1}^{\infty}$ and $c\in L_{-\tau-2}^{\infty}$, then
\begin{equation*}
\|(L-\Delta_{\mathbb{E}})(u-u_{R})\|_{L_{-\tau-2}^{q}(M\setminus B_{R})}
\leq CR^{-\tau}\|u-u_{R}\|_{W_{-\tau}^{2,q}(M\setminus B_{R})}.
\end{equation*}
Substituting this into \eqref{Fredholm eqn 5} and choosing $R$ sufficiently large (this fixes the value of $R$),
\begin{equation}\label{Fredholm eqn 6}
\|u-u_{R}\|_{W_{-\tau}^{2,q}(M\setminus B_{R})} \leq C\|L(u-u_{R})\|_{L_{-\tau-2}^{q}(M\setminus B_{R})}
\end{equation}
Using $u-u_{R}\equiv0$ in $B_{R}$, \eqref{Fredholm eqn 6} becomes
\[
\|u-u_{R}\|_{W_{-\tau}^{2,q}} \leq C\|L(u-u_{R})\|_{L_{-\tau-2}^{q}}
\]
and so
\begin{equation}\label{Fredholm eqn 7}
\begin{split}
\|u\|_{W_{-\tau}^{2,q}} \leq {} & C\|Lu\|_{L_{-\tau-2}^{q}}+C\|Lu_{R}\|_{L_{-\tau-2}^{q}}+\|u_{R}\|_{W_{-\tau}^{2,q}} \\
\leq {} & C\|Lu\|_{L_{-\tau-2}^{q}}+C\|L(\eta_{R}u)\|_{L_{-\tau-2}^{q}(B_{2R})}+C\|\eta_{R}u\|_{W_{-\tau}^{2,q}(B_{2R})} \\
\leq {} & C\|Lu\|_{L_{-\tau-2}^{q}}+C\|u\|_{W^{2,q}(B_{2R})}.
\end{split}
\end{equation}
Applying the $L^{q}$-estimate \cite[Theorem 9.11]{GilbargTrudinger2001} to $u$ on $B_{2R}$, we obtain
\begin{equation}\label{Fredholm eqn 8}
\|u\|_{W^{2,q}(B_{2R})} \leq \|Lu\|_{L^{q}(B_{4R})}+\|u\|_{L^{q}(B_{4R})}
\leq C\|Lu\|_{L_{-\tau-2}^{q}}+\|u\|_{L^{q}(B_{4R})}.
\end{equation}
Combining \eqref{Fredholm eqn 7} and \eqref{Fredholm eqn 8}, we obtain the required inequality for $\Omega=B_{4R}$.

\bigskip
\noindent
{\bf Step 2.} $\mathrm{Ker}L$ is finite dimensional.
\bigskip

Since $L$ is continuous, then $\mathrm{Ker}L$ is closed. To prove Step 2, it suffices to show that the set
\[
S = \{u\in W_{-\tau}^{2,p}:\text{$Lu=0$ and $\|u\|_{W_{-\tau}^{2,p}}=1$}\}
\]
is compact, or equivalently, every sequence $\{u_{k}\}_{k=1}^{\infty}\subset S$ has a Cauchy subsequence. By the Sobolev embedding, there exists a subsequence $\{u_{k_{i}}\}_{i=1}^{\infty}$ such that
\begin{equation}\label{Fredholm eqn 1}
\|u_{k_{i}}-u_{k_{j}}\|_{L^{q}(\Omega)} \rightarrow 0 \ \text{as $i,j\rightarrow\infty$}.
\end{equation}
Applying Step 1 to $(u_{k_{i}}-u_{k_{j}})$ and using \eqref{Fredholm eqn 1}, $\{u_{k_{i}}\}_{i=1}^{\infty}$ is a Cauchy sequence in $W_{-\tau}^{2,q}$.

\bigskip
\noindent
{\bf Step 3.} The range of $L$ is closed.
\bigskip

Since $\mathrm{Ker}L$ is finite dimensional, then there exists a closed subspace $Z$ such that $W_{-\tau}^{2,q}=\mathrm{Ker}L+Z$ and $\mathrm{Ker}L\cap Z=\{0\}$. We claim that
\begin{equation}\label{Fredholm claim}
\|u\|_{W_{-\tau}^{2,q}} \leq C\|Lu\|_{L_{-\tau-2}^{q}} \ \text{for $u\in Z$}.
\end{equation}
This can be proved by the contradiction argument. If \eqref{Fredholm claim} does not hold, then there exists a sequence $\{u_{k}\}_{k=1}^{\infty}\subset Z$ is a sequence such that
\begin{equation}\label{Fredholm eqn 2}
\|u_{k}\|_{W_{-\tau}^{2,q}} = 1 \ \text{and} \ \|Lu_{k}\|_{L_{-\tau-2}^{q}}\rightarrow0 \ \text{as $k\rightarrow\infty$}.
\end{equation}
Applying the similar argument of Step 2, there exists a Cauchy subsequence $\{u_{k_{i}}\}_{i=1}^{\infty}\subset Z$ in $W_{-\tau}^{2,q}$. Denote its limit by $u_{\infty}$. Then \eqref{Fredholm eqn 2} shows $\|u_{\infty}\|_{W_{-\tau}^{2,q}}=1$ and $u_{\infty}\in\mathrm{Ker}L\cap Z$, which contradicts with $\mathrm{Ker}L\cap Z=\{0\}$.

Next we prove $\mathrm{Ran}L$ is closed. Let $h_{\infty}$ be a limit point of $\mathrm{Ran}L$. Then there exist two sequences $\{v_{i}\}_{i=1}^{\infty}\subset W_{-\tau}^{2,q}$ and $\{h_{i}\}_{i=1}^{\infty}\subset L_{-\tau-2}^{q}$ are two sequences such that
\begin{equation}\label{Fredholm eqn 3}
Lv_{i} = h_{i} \ \text{and} \ h_{i}\rightarrow h_{\infty} \ \text{as $i\rightarrow\infty$}.
\end{equation}
It suffices to show that $h_{\infty}\in\mathrm{Ran}L$. Since $W_{-\tau}^{2,q}=\mathrm{Ker}L+Z$, there exists $w_{i}\in Z$ such that
\begin{equation}\label{Fredholm eqn 4}
Lw_{i} = Lv_{i} = h_{i}.
\end{equation}
Applying \eqref{Fredholm claim} to $(w_{i}-w_{j})$, we see that $\{w_{i}\}$ is a Cauchy sequence in $W_{-\tau}^{2,q}$. Denote its limit by $w_{\infty}$. Then \eqref{Fredholm eqn 3} and \eqref{Fredholm eqn 4} show $Lw_{\infty}=h_{\infty}$ and so $h_{\infty}\in\mathrm{Ran}L$.

\bigskip
\noindent
{\bf Step 4.} $\mathrm{Coker}L$ is finite dimensional.
\bigskip

We use $L^{*}:(L^q_{-\tau-2})^{*}\to(W^{2,q}_{-\tau})^{*}$ to denote the $L^2$-adjoint of $L$ with respect to $\mathrm{dvol}_{g}$. Using $(L^q_{-\tau-2})^{*}=L_{\tau+2-n}^{q^{*}}$ and Theorem \ref{elliptic theory}, we obtain $\mathrm{Ker}L^{*}\subset W_{\tau+2-n}^{2,q^{*}}$. Write $c^{*}=c-\Div_g X\in L_{-\tau-2}^{\infty}$. Direct calculation shows
\[
L^{*}u = \Delta_g u-\langle\mathrm du,X\rangle_g+c^{*}u \ \text{for $u\in W_{-\tau}^{2,q}$}.
\]
Repeating the argument of Step 1 and 2 for the operator $L^{*}$, we obtain $\mathrm{Ker}L^{*}$ is also finite dimensional. Then $(\mathrm{Coker}L)^{*}\cong\mathrm{Ker}L^{*}$ implies
\[
\dim\mathrm{Coker}L = \dim\mathrm{Ker}L^{*} < \infty.
\]

The proof of Theorem \ref{Fredholm} is now completed.
\end{proof}

\end{document}